\documentclass{article}

\author{Ofir Gorodetsky} \title{A polynomial analogue of Landau's theorem and related problems}
\date{}

\usepackage[margin=1in]{geometry}
\usepackage[all]{xy}
\usepackage{amssymb}
\usepackage{amsfonts}
\usepackage{amsthm}
\usepackage{amsmath}
\usepackage{hyperref}
\usepackage{mathtools}

\newtheorem*{thm*}{Theorem}
\newtheorem{thm}{Theorem}[section]

\newtheorem{lem}[thm]{Lemma}  
\newtheorem{proposition}[thm]{Proposition}
\newtheorem{cor}[thm]{Corollary}
\theoremstyle{definition}

\newtheorem{remark}[thm]{Remark}

\newcommand{\Mod}[1]{\ \text{mod}\ #1}

\newcommand{\FF}{\mathbb{F}}

\numberwithin{equation}{section}

\begin{document}
\maketitle
\begin{abstract}
Recently, an analogue over $\FF_q[T]$ of Landau's theorem on sums of two squares was considered by Bary-Soroker, Smilansky and Wolf. They counted the number of monic polynomials in $\FF_q[T]$ of degree $n$ of the form $A^2+TB^2$, which we denote by $B(n,q)$. They studied $B(n,q)$ in two limits: fixed $n$ and large $q$; and fixed $q$ and large $n$. We generalize their result to the most general limit $q^n \to \infty$. More precisely, we prove
\begin{equation*}
B(n,q) \sim K_q \cdot \binom{n-\frac{1}{2}}{n} \cdot q^n , \qquad q^n \to \infty,
\end{equation*}
for an explicit constant $K_q = 1+O\left(1/q\right)$. Our methods are different and are based on giving explicit bounds on the coefficients of generating functions. These methods also apply to other problems, related to polynomials with prime factors of even degree.
\end{abstract}

\section{Introduction}
Let $b(n)$ be the characteristic function of integers that are representable as a sum of two squares and let
\begin{equation}
B(x) = \sum_{n \le x} b(n)
\end{equation}
be the number of such integers up to $x$. Landau's theorem \cite{land} gives an asymptotic formula for $B(x)$:
\begin{equation}
B(x) = K\frac{x}{\sqrt{\ln x}} + O \left(\frac{x}{\ln^{3/2} x} \right), \qquad x \to \infty,
\end{equation}
where
\begin{equation}
K = \frac{1}{\sqrt{2}} \prod_{p \equiv 3 \bmod 4}(1-p^{-2})^{-1/2} \approx 0.764
\end{equation}
is the Landau-Ramanujan constant.

A polynomial analogue of the function $B(x)$, which we describe below, was studied by Bary-Soroker, Smilansky and Wolf \cite{lior}. Let $q$ be an odd prime power, and $\FF_q$ denote the field of $q$ elements. Denote by $\mathcal{M}_{n,q}$ the set of monic polynomials in $\FF_q[T]$ of degree $n$, by $\mathcal{M}_{q}=\cup_{n \ge 0} \mathcal{M}_{n,q}$ the set of all monic polynomials in $\FF_q[T]$ and by $\mathcal{P}_{q}$ the set of monic irreducible polynomials in $\FF_q[T]$.

For a polynomial $f \in \mathcal{M}_{n,q}$ we define the characteristic function
\begin{equation}
b_q(f) = \begin{cases} 1, & f=A^2+TB^2 \text{ for }A,B\in \FF_q[T], \\ 0& \text{otherwise}, \end{cases}
\end{equation}
and the counting function
\begin{equation}
B(n,q)=\sum_{f \in \mathcal{M}_{n,q}}b_q(f).
\end{equation}
In \cite[Thms.~1.2 and 1.3]{lior}, the asymptotic behaviour of $B(n,q)$ was studied in two  limits, fixed $n$ and large $q$  and fixed $q$ and large $n$,
\begin{align}\label{eq:bigq}
B(n,q) &= \frac{\binom{2n}{n}}{4^n} \cdot q^n + O_n(q^{n-1}), \qquad q \to \infty,\\
B(n,q) &= \frac{K_q}{\sqrt{\pi}}\cdot 	\frac{q^n}{\sqrt{n}} + O_q\left(\frac{q^n}{n^{3/2}}\right), \qquad n \to \infty,\label{eq:bign}
\end{align}
where
\begin{align}\label{eq:constantlior}
K_q&=(1-q^{-1})^{-\frac{1}{2}}\prod_{P \in \mathcal{P}_q:(P/T)=-1}(1-q^{-2\deg P})^{-\frac{1}{2}} = 1+O\left(\frac{1}{q}\right).
\end{align}
The asymptotic formula \eqref{eq:bigq} is proved using elementary combinatorial considerations, while \eqref{eq:bign} is proved using complex analysis. In a recent preprint, Matei \cite{matei} extended \eqref{eq:bigq} by writing $B(n,q)$ as a polynomial in $q$ and proving that the coefficients have a certain structure and a cohomological interpretation.

The authors of \cite{lior} write ``We do not know of an asymptotic formula for $B(n,q)$ in any more general sub-limits
of $q^n \to \infty$''. We prove an asymptotic formula in the most general limit $q^n \to \infty$.
\begin{thm}\label{landauthm}
\begin{equation}
B(n,q) =K_q \cdot \binom{n-\frac{1}{2}}{n} \cdot q^n +O\left( \frac{q^{n-1}}{n^{3/2}} \right), \qquad q^n \to \infty,
\end{equation}
and the implied constant is absolute.
\end{thm}
We also investigate the constant $K_q$, which turns out to be an analytic function of $q^{-1}$.
\begin{remark}
To see that Theorem \ref{landauthm} is consistent with \eqref{eq:bigq} and \eqref{eq:bign}, note that 
\begin{equation}
\binom{n-\frac{1}{2}}{n} = \frac{\binom{2n}{n}}{4^n} \sim \frac{1}{\sqrt{\pi n}}, \qquad n \to \infty,
\end{equation}
and that by \eqref{eq:constantlior}, $K_q=1+O\left( 1/q \right)$.
\end{remark}
Our methods differ from those used in \cite{lior}, and are based on explicit estimates for the coefficients of a certain class of generating functions. Another advantage of our methods, is that they apply to other problems, see \S \ref{furthapp}.

\section{Further applications}\label{furthapp}
\subsection{Prime factors of even degree}
Suppose we want to count ``ordinary'' sums of two squares in $\mathbb{F}_q[T]$, i.e. elements of the form $A^2+B^2$. If $q \equiv 1 \bmod 4$, then $\sqrt{-1} \in \FF_q$. Given $f \in \FF_q[T]$, we have $f=((f-1)/(2\sqrt{-1}))^2 + ((f+1)/2)^2$, and so every polynomial is of the form $A^2+B^2$. If $q \equiv 3 \mod 4$, Leahey \cite{leahey} has shown that the polynomials of the form $A^2+B^2$ are those whose prime factors of odd degree appear only with even multiplicity. Imitating Leahey's proof, one sees that if $q$ is an odd prime power and 
\begin{equation}
\alpha \in \FF_q^{\times} \setminus (\FF_q^{\times})^2,
\end{equation}
then the monic polynomials of the form $A^2-\alpha  B^2$ ($A,B \in \FF_q[T]$) are also characterised by the property that their prime factors of odd degree appear with even multiplicity.
This raises the problem of estimating the number of monic polynomials of a given degree in the following set, which makes sense for any prime power $q$:
\begin{equation}\label{s1def}
S_1(q) = \{ f \in \mathcal{M}_{q} \, : P \mid f \text{ and } 2\nmid \deg P \implies 2 \mid v_P(f) \},
\end{equation}
where $v_P(f)$ is the multiplicity of $P$ in $f$. Let
\begin{equation}\label{b1def}
B_1(2n,q) = \# \{ f \in S_{1}(q) : \deg f = 2n \}.
\end{equation}
Chuang, Kuan and Yu \cite{3proc}
considered two natural variations on $S_1(q)$. The first is the subset of $S_1(q)$ of polynomials with no odd-degree prime factors:
\begin{equation}
\begin{split}
S_2(q) &= \{ f \in \mathcal{M}_{q} \, : P \mid f \implies 2 \mid \deg P \}, \\
B_2(2n,q) &= \# \{ f \in S_{2}(q) : \deg f = 2n \}.  
\end{split}
\end{equation}
The second is the subset of squarefree polynomials in $S_2(q)$:
\begin{equation}\label{b3count}
\begin{split}
S_3(q) &= \{ f \in \mathcal{M}_{q} \, : P \mid f \implies 2 \mid \deg P \text{ and } P^2 \nmid f \}, \\
B_3(2n,q) &= \# \{ f \in S_{3}(q) \,: \deg f = 2n \}.  
\end{split}
\end{equation}
The motivation for studying $S_2(q)$, $S_3(q)$ is the following. Assume that $q$ is odd. As observed by Artin \cite{artin} in his study of quadratic extensions of $\FF_q(T)$, the analogue over $\FF_q[T]$ of a fundamental discriminant is a squarefree monic polynomial $D$. The negative Pell equation asks for the solubility of
\begin{equation}\label{negpell}
X^2-DY^2=\gamma_q\text{ with }X,Y \in \FF_q[T],
\end{equation}
where $\gamma_q$ is a generator of $\FF_q^{\times}$. By considering \eqref{negpell} modulo a prime factor $P$ of $D$, we find
\begin{equation}\label{quadrec}
\left( \frac{\gamma_q}{P} \right) = 1,
\end{equation}
where $\left(\frac{\bullet}{P} \right)$ is the Legendre symbol modulo $P$. By quadratic reciprocity, equation \eqref{quadrec} implies that $P$ has even degree. Thus, the negative Pell equation has no solution for a given fundamental discriminant $D$ of degree $2N$ unless $D$ is among those $B_3(2N,q)$ polynomials counted in \eqref{b3count}. 

The problem of estimating, in the limit $n \to \infty$, the proportion in $B_3(2n,q)$ of the discriminants $D$ for which \eqref{negpell} is soluble, is an open problem studied by Bae and Jung \cite[Thm.~1.5]{bae}. Their work is motivated by the number field setting, studied by Stevenhagen \cite[Conj.~1.2]{steven} and Fouvry and Kl\"uners \cite[Thm.~1]{fouvry}.

The asymptotics of $B_2(2n,q)$, $B_3(2n,q)$ in the limit $n \to \infty$ are given in \cite[Thms.~1 and 2]{3proc}:
\begin{equation}\label{3procfirst}
\begin{split}
B_2(2n,q) &= C_{q,2} \frac{q^{2n}}{\sqrt{\pi n}} +O_q\left(\frac{q^{2n}}{n^{3/2}}\right), \qquad n \to \infty, \\
B_3(2n,q) &= C_{q,3} \frac{q^{2n}}{\sqrt{\pi n}} +O_q\left(\frac{q^{2n}}{n^{3/2}}\right), \qquad n \to \infty,
\end{split}
\end{equation}
where $C_{q,2}$ and $C_{q,3}$ are positive constants, given explicitly in \cite[Thms.~1 and 2]{3proc} as infinite products.

We modify the main term of \eqref{3procfirst} to obtain an asymptotic formula in the general limit $q^n \to \infty$.
\begin{thm}\label{thmchinintimprovquad}
For any positive integer $n$,
\begin{align}\label{eq:b1est}
B_1(2n,q) &= C_{q,1} \cdot \binom{n-\frac{1}{2}}{n} \cdot q^{2n} +O\left( \frac{q^{2n-1}}{n^{3/2}} \right), \qquad q^n \to \infty,\\
B_2(2n,q)&=  C_{q,2} \cdot \binom{n-\frac{1}{2}}{n} \cdot q^{2n} +O\left( \frac{q^{2n-1}}{n^{3/2}} \right), \qquad q^n \to \infty, \label{eq:b2est}\\
B_3(2n,q)&= C_{q,3} \cdot \binom{n-\frac{1}{2}}{n} \cdot q^{2n} +O\left( \frac{q^{2n-1}}{n^{3/2}} \right), \qquad q^n \to \infty,,\label{eq:b3est}
\end{align}
where the implied constants are absolute. Moreover, $C_{q,i} = 1 + O\left( 1/q \right)$.
\end{thm}
\subsection{Higher genus}
We consider a higher genus analogue of the sets $S_1(q)$, $S_2(q)$ and $S_3(q)$. Consider a function field $K$ with finite constant field $\FF_q$, i.e. $K$ is a finitely generated field extension of transcendence degree one over $\FF_q$ and $\FF_q$ is algebraically closed in $K$. Let $\mathfrak{Q}$ be a fixed place of $K$. Let  $\mathbb{P}_K$ be the set of places of $K$.

For each place $\mathfrak{P}$ of $K$, the corresponding valuation $K \to \mathbb{Z}\cup \{ \infty \}$ is denoted by $v_{\mathfrak{P}}$. The residue field at $\mathfrak{P}$ is denoted by $\FF_{\mathfrak{P}}$. The degree $\deg_K {\mathfrak{P}}$ of $\mathfrak{P}$ is $[\FF_{\mathfrak{P}} : \FF_q]$. Let $\text{Div}(K)$ be the divisor group of $K$. Every element $D$ in $\text{Div}(K)$ may be written uniquely as
\begin{equation}
D = \sum_{\mathfrak{P} \in \mathbb{P}_K} v_{\mathfrak{P}}(D) \cdot \mathfrak{P} .
\end{equation}
Given $D \in \text{Div}(K)$, we define the degree of $D$ as
\begin{equation}
\deg D = \sum_{\mathfrak{P} \in \mathbb{P}_K}  v_{\mathfrak{P}}(D) \cdot \deg_K \mathfrak{P}.
\end{equation}
The set
\begin{equation}
\text{Div}_{0,\mathfrak{Q}}(K) = \{ D \in \text{Div}(k) : \deg D=0,\quad v_{\mathfrak{P}}(D) < 0 \implies \mathfrak{P}=\mathfrak{Q} \} 
\end{equation}
is an higher genus analogue of the set $\mathcal{M}_q \subseteq \FF_q[T]$: If $K=\FF_q(T)$ and $\mathfrak{Q}$ is the prime at $1/T$, then the map $\mathcal{M}_q \to \text{Div}_{0,\mathfrak{Q}}(K)$ given by $f \mapsto \text{div}(f)$ is a bijection. Similarly, the function $\deg_{\mathfrak{Q}}: \text{Div}_{0,\mathfrak{Q}}(K) \to \mathbb{Z}$ given by
\begin{equation}\label{degqdef}
\deg_{\mathfrak{Q}}\bigg(\sum_{\mathfrak{P}} n_{\mathfrak{P}} \cdot \mathfrak{P}\bigg) = \sum_{\mathfrak{P}:n_{\mathfrak{P}} < 0} (-n_{\mathfrak{P}}) \cdot \deg \mathfrak{P} = -n_{\mathfrak{Q}} \cdot \deg \mathfrak{Q}
\end{equation}
is an higher genus analogue of a degree of a polynomial.

Given positive integers $n \ge 1$, $\ell \ge 1$, $r\ge 2$, we consider the following sets of divisors, which are analogues of $S_1(q)$, $S_2(q)$ and $S_3(q)$:
\begin{align}
S_1(r,K)&=\{D \in \text{Div}_{0,\mathfrak{Q}}(K)\,:  v_{\mathfrak{P}}(D) >0 \implies \frac{r}{\gcd(r, \deg_K \mathfrak{P})} \mid v_{\mathfrak{P}}(D) \},\\
S_2(r,K)&=\{D \in \text{Div}_{0,\mathfrak{Q}}(K)\,:  v_{\mathfrak{P}}(D) >0 \implies r \mid \deg_K \mathfrak{P} \},\\
S_3(r,\ell,  K)&=\{D \in \text{Div}_{0,\mathfrak{Q}}(K)\,:  v_{\mathfrak{P}}(D) >0 \implies r \mid \deg_K \mathfrak{P} \text{ and }v_{\mathfrak{P}}(D)  \le \ell \}.
\end{align}
We elaborate on the algebraic meaning of the sets $S_1$, $S_2$, $S_3$. Let $K_r$ be the fixed constant field extension over $K$ of degree $r$. We have the following relation between primes of $K_r$ and $K$.
\begin{proposition}\cite[Prop.~8.13]{rosen}\label{krsplit}
Let $\mathfrak{P}$ be a prime of $K$. Then $\mathfrak{P}$ splits into $\gcd(r,\deg_K \mathfrak{P} )$ primes in $K_r$. Let $\mathfrak{P}_r$ be a prime of $K_r$ lying over $\mathfrak{P}$. Denote by $f(\mathfrak{P}_r / \mathfrak{P})$ the relative degree of $\mathfrak{P}_r$  over $\mathfrak{P}$. Then
\begin{equation}
\deg_{K_r} \mathfrak{P}_r = \frac{\deg_{K} \mathfrak{P}}{\gcd(r, \deg_{K} \mathfrak{P})}, \qquad f(\mathfrak{P}_r / \mathfrak{P})=\frac{r}{	\gcd(r, \deg_{K} \mathfrak{P})}.
\end{equation}
\end{proposition}
In particular, primes of $K$ of degree divisible by $r$ are exactly those that split completely in $K_r$, which gives an algebraic meaning to $S_2$ and $S_3$. We may define the (linear) norm map \cite[Chap.~IV.7]{chev}
\begin{equation}
N_{K_r/K}: \text{Div}(K_r) \to \text{Div}(K)
\end{equation}
by 
\begin{equation}
\mathfrak{P}_r \mapsto f(\mathfrak{P}_r / \mathfrak{P}) \cdot \mathfrak{P},
\end{equation}
where $\mathfrak{P}_r$ is a prime of $K_r$ lying over $\mathfrak{P}$, a prime of $K$.
By Proposition \ref{krsplit}, the image of $N_{K_r/K}$ is spanned by 
\begin{equation}
\left\{ \frac{r}{\gcd(r,\deg_K \mathfrak{P})}  \mathfrak{P} :  \mathfrak{P} \in \mathbb{P}_K \right\}.
\end{equation}
Intersecting the image $N_{K_r/K}(K_r)$ with $\text{Div}_{0,\mathfrak{Q}}(K)$, we recover the set $S_1(r,K)$. When $r=2, K = \FF_q(T)$ and $\mathfrak{Q}$ is the prime at $1/T$, we obtain $S_1(r,K) = S_1(q)$. Hence, counting divisors of $S_1(r,K)$ is a generalization of Landau's problem over $\FF_q[T]$. The asymptotics of
\begin{equation}
B_i(rn,r,K) = \# \{ D \in S_i(r,K)\,:  \deg_{\mathfrak{Q}} (D) = rn \} \quad (i=1,2)
\end{equation}
and
\begin{equation}
B_3(rn,r,\ell,K) = \# \{ D \in S_3(r,\ell,K) \,:  \deg_{\mathfrak{Q}} (D) = rn \}
\end{equation}
in a very general range of parameters are given in the following theorem.
\begin{thm}\label{thmchinintimprov}
Let $K$ be a function field with constant field $\FF_q$. Fix a prime $\mathfrak{Q}$ of $K$. Let $g_K$ be the genus of $K$, and define 
\begin{equation}
M_{K,\mathfrak{Q}} = \max\{ g_K, \deg \mathfrak{Q} \}.
\end{equation}
Let $n \ge 1$, $\ell \ge 1$, $r \ge 2$ be given integers. If $\deg_{K} \mathfrak{Q} \nmid rn$, we have
\begin{equation}
B_1(rn,r,K)=B_2(rn,r,K)=B_3(rn,r,\ell,K)=0.
\end{equation}
Otherwise, we have
\begin{align}
B_1(rn,r,K)&= \binom{n+\frac{1}{r}-1}{n} \cdot q^{rn} \cdot \left( C_{1,r,K}  + O\left(\frac{\exp(150 \frac{M_{K,\mathfrak{Q}}}{rq^{r/2}}) \frac{M_{K,\mathfrak{Q}}}{rq^{r/2}}}{n}\right) \right),\label{eq:b11est}\\
B_2(rn,r,K)&= \binom{n+\frac{1}{r}-1}{n} \cdot q^{rn} \cdot \left( C_{2,r,K}  + O\left(\frac{\exp(150 \frac{M_{K,\mathfrak{Q}}}{rq^{r/2}}) \frac{M_{K,\mathfrak{Q}}}{rq^{r/2}}}{n}\right) \right),\label{eq:b22est}\\
B_3(rn,r,\ell,K) &= \binom{n+\frac{1}{r}-1}{n} \cdot q^{rn}\cdot \left( C_{r,\ell,K}  + O\left(\frac{\exp(150 \frac{M_{K,\mathfrak{Q}}}{rq^{r/2}}) \frac{M_{K,\mathfrak{Q}}}{rq^{r/2}}}{n}\right) \right),\label{eq:b33est}
\end{align}
as long as 
\begin{equation}\label{rangeofexp}
n \gg 1 + \frac{M_{K,\mathfrak{Q}}}{r^2\ln(q)} \ln \left(\frac{M_{K,\mathfrak{Q}}}{r^2\ln(q)} + 1\right).
\end{equation}
In addition, the constants $C_{1,r,K},C_{2,r,K}, C_{r,\ell,K}$ may be estimated by
\begin{equation}
C_{1,r,K},C_{2,r,K},C_{r,\ell,K} = \exp\left( O\left( \frac{M_{K,\mathfrak{Q}}}{rq^{r/2}} \right) \right).
\end{equation}
\end{thm}
A variant of the functions $B_2(rn,r,K)$, $B_3(rn,r,\ell,K)$, where a prime $\mathfrak{Q}$ was not included in the definition, was studied in \cite[Thm.~4]{3proc}. The results obtained there are meaningful only in the limit $n \to \infty$.

We also investigate the constants $C_{1,r,K}$, $C_{2,r,K}$, $C_{r,\ell,K}$ and provide an expression for them, involving an infinite product over primes of $\mathbb{P}_K$.
\begin{remark}
In the limit $n \to \infty$, the expression $\binom{n+\frac{1}{r}-1}{n}$ appearing in Theorem \ref{thmchinintimprov} may be replaced with $\frac{n^{\frac{1}{r}-1}}{\Gamma(\frac{1}{r})} \left( 1 + O \left( \frac{1}{rn} \right) \right)$: by \cite[Eq.~14]{gamma},  
\begin{equation}
\begin{split}
\binom{n+\frac{1}{r}-1}{n} 
&\le \frac{n^{\frac{1}{r}-1}}{\Gamma(\frac{1}{r})}, \quad \mbox{and}\\
\binom{n+\frac{1}{r}-1}{n} 
&\ge \frac{(n+\frac{1}{r})^{\frac{1}{r}-1}}{\Gamma(\frac{1}{r})} = \frac{n^{\frac{1}{r}-1}}{\Gamma(\frac{1}{r})} \left(1+\frac{1}{rn}\right)^{\frac{1}{r}-1}\\
& \ge \frac{n^{\frac{1}{r}-1}}{\Gamma(\frac{1}{r})} \left(1+\frac{1}{rn}\right)^{-1} \ge \frac{n^{\frac{1}{r}-1}}{\Gamma(\frac{1}{r})} \left(1-\frac{1}{rn}\right).
\end{split}
\end{equation}
\end{remark}

\subsection{Prime factors in an arithmetic progression}
Let $a,m \in \FF_q[T]$ be a pair of coprime polynomials with $m$ monic of positive degree. Let $G_{a,m} \subseteq \FF_q[T]$ be the set of monic polynomials whose monic prime factors are in the arithmetic progression $a(T) \bmod m(T)$. Let 
\begin{equation}\label{snamdef}
S(n; a, m) = \# \{ f \in G_{a,m} \,: \deg (f) = n \},
\end{equation}
be the number of polynomials of degree $n$ in the semigroup $G_{a,m}$. What bounds are satisfied by $S(n;a,m)$?

Although studying $S(n;a,m)$ is interesting for its own sake, the case $a=1$ is an important ingredient in the proof of the following theorem of Thorne \cite[Thm.~1.2]{thorne}, which is a function field analogue of a theorem of Shiu \cite[Thm.~1]{shiu}.
\begin{thm}[Thorne]\label{thornethm}
Let $a,m \in \FF_q[T]$ with $\gcd(a, m) = 1$ and $\deg m>0$, $m$ monic. There exists a constant $D'$ (depending on $q$ and $m$) such that for any $D > D'$ there exists a string of consecutive monic primes \begin{equation}
p_{r+1} \equiv p_{r+2} \cdots \equiv p_{r+k} \equiv a \bmod m
\end{equation}
of degree at most $D$, where $k$ satisfies
\begin{equation}
k \gg_{q} \frac{1}{\phi(m)} \left(\frac{\log D}{(\log \log D)^2} \right)^{1/\phi(m)}.
\end{equation}
\end{thm}
Here ``consecutive'' is to be understood with respect to lexicographic order. Thorne's proof uses the following asymptotic formula, corollary of a result of Manstavi\v cius and Skrabut\. enas \cite[Thm.~1]{1mod}:
\begin{align}
\label{eq:arithintro}
S(n;1,m) &\sim  C_{1,m} \cdot \frac{n^{-1+\frac{1}{\phi(m)}}}{\Gamma(1/\phi(m))} \cdot  q^n , \quad n \to \infty,
\end{align}
where
\begin{align}
C_{a,m} &= \lim_{u \to \frac{1}{q}^{-}} (1-qu)^{1/\phi(m)} \prod_{P \in \mathcal{P}_q:P\equiv a\Mod m}(1-u^{\deg P})^{-1}.
\end{align} 
As remarked after \cite[Lem.~3.5]{thorne}, making \eqref{eq:arithintro} effective allows one to determine the constant $D'$ appearing in Theorem \ref{thornethm}. We establish the following version of \eqref{eq:arithintro}.
\begin{thm}\label{semiariththm}
Let $a,m \in \FF_q[T]$ be a pair of coprime polynomials with $\deg m>0$ and $m$ monic. For 
\begin{equation}
n \ge 1+5\left(6 \deg m +30\right)\ln\left(6 \deg m + 30\right),
\end{equation}
we have
\begin{equation}\label{snamest}
S(n;a,m) =  \binom{n+\frac{1}{\phi(m)}-1}{n} \cdot q^n  \left( C_{a,m} + O\left(\frac{\exp(\frac{3(\deg m+3)}{\sqrt{q}})\frac{(\deg m+3)}{\sqrt{q}}}{n}\right) \right),
\end{equation}
where the implied constant is (at most) 48.
\end{thm}

\section{Methods}\label{methods}
Let $q$ be a prime power, and let $\FF_q[T]$ be the polynomial ring over the finite field $\FF_q$. Denote by $\pi_q(n)$ the number of irreducible monic polynomials in $\FF_q[T]$ of degree $n$.

Let $G \subseteq \mathbb{F}_{q}[T]$ be a multiplicative semigroup generated by (possibly infinitely many) mutually coprime monic polynomials. Let $g(n)$ be the number of generators of degree $n$. The coefficients of the generating function
\begin{equation}
\begin{split}
F_G(x) &:= \prod_{n\ge 1} (1+x^n+x^{2n} + \cdots )^{g(n)} \\
&=\prod_{n \ge 1} (1-x^n)^{-g(n)}
\end{split}
\end{equation}
count how many polynomials in $G$ are of a specified degree. By writing $F_G(x)=\exp( \ln F_G(x) )$, we may express $F_G(x)$ as the formal power series 
\begin{equation}
\exp \left(\sum_{n \ge 1} \frac{\psi_G(n) x^n}{n} \right),
\end{equation}
where
\begin{equation}
\psi_G(n):= \sum_{d \mid n} d \cdot g(d).
\end{equation}
For instance, when $G=\mathcal{M}_q \subseteq \FF_q[T]$ is the semigroup of monic polynomials, generated by monic irreducible polynomials, the corresponding function $\psi_{\mathcal{M}_q}$ is the weighted prime-counting function $\psi_{\mathcal{M}_q}(n) = \sum_{d \mid n} d \cdot \pi_q(d)$, which satisfies \cite[Prop.~2.1]{rosen}
\begin{equation}\label{classicpsi}
\psi_{\mathcal{M}_q}(n) = \sum_{d \mid n} d\cdot \pi_q(d) = q^n.
\end{equation}
It turns out that for many semigroups that occur in counting problems, $\psi_G$ satisfies the following bound, where $\alpha, \beta,c_1,c_2$ are some real numbers satisfying $\alpha > \beta>0$:
\begin{align}\label{eq:prop1}
|\psi_G(n) - c_1 \beta^{-n}| &\le c_2 \alpha^{-n}.
\end{align}
The bound \eqref{eq:prop1} insures that we can write $F_G$ as a product of two analytic functions, with radii of convergence at most $\alpha$ and exactly $\beta$, respectively:
\begin{equation}\label{eq:decompprod}
\begin{split}
F_G(x) &=\exp \left(\sum_{n \ge 1} \frac{e_n x^n}{n} \right) \cdot \exp \left(\sum_{n \ge 1} \frac{c_1 \beta^{-n} x^n}{n} \right)\\
& = \exp \left(\sum_{n \ge 1} \frac{e_n x^n}{n}\right) \cdot \left(1-\frac{x}{\beta}\right)^{-c_1},
\end{split}
\end{equation}
where 
\begin{equation}
e_n=\psi_G(n) - c_1 \beta^{-n}.
\end{equation}
In particular, the radii of convergence are distinct.

The problem of estimating the coefficients of a product of two functions with different radii of convergence has been studied previously. First, we introduce a useful piece of notation. For a power series $f(x)=\sum_{i \ge 0} f_i x^i$, $[x^n]f(x)$ denotes the coefficient of $x^n$ in $f$, i.e. $f_n$.

We mention two classical results. The first is \cite[Thm.~VI.12]{flaj}.
\begin{thm}\label{analy1} Let $a(x) =\sum a_n x^n$ and $b(x) = \sum b_n x^n$ be two power series with radii of convergence $\alpha > \beta \ge 0$, respectively. Assume that $b_{n-1}/b_n \to \beta$ as $n$ tends to $\infty$, and that $a(\beta) \neq 0$. Then the coefficients of the product $f(x) =a(x)\cdot b(x)$ satisfy the following, as $n$ tends to $\infty$:
\begin{equation}
[x^n] f(x) \sim a(\beta) \cdot b_n.
\end{equation}
\end{thm}
Theorem \ref{analy1} is proved using calculus. Applying Theorem \ref{analy1}, we solve the problem of estimating $[x^n]F_G(x)$ as $n$ tends to $\infty$ and $q$ is fixed, under assumption \eqref{eq:prop1}:
\begin{equation}
\begin{split}
[x^n]F_G(x) &\sim \exp\left(\sum_{n \ge 1} \frac{e_n \beta^n}{n} \right)  \cdot  \binom{-c_1}{n} \cdot \left(-\frac{1}{\beta}\right)^n\\
&=\exp\left(\sum_{n \ge 1} \frac{e_n \beta^n}{n} \right) \cdot \binom{n+c_1-1}{n} \cdot \beta^{-n} .
\end{split}
\end{equation}
Theorem \ref{analy1} does not give information on the rate of convergence. A method, due to Darboux, gives a more informative estimate for certain $b$ \cite[Thm.~11.10b]{henrici}.
\begin{thm}\label{darbmetthm}
Let $a(x) = \sum a_n x^n$ and $b(x) =(1-\frac{x}{\beta})^{-c_1} = \sum b_n x^n$ $(c_1 \in \mathbb{C} \setminus \mathbb{Z})$ be two power series with radii of convergence $\alpha > \beta \ge 0$, respectively. Fix an integer $m \ge 0$. Then the coefficients of the product $f(x) =a(x)\cdot b(x)$ satisfy the following, as $n$ tends to $\infty$:
\begin{align}\label{eq:darbstat}
[x^n] f(x) &= b_n \cdot \left( a(\beta) + \sum_{k=1}^{m} \frac{\binom{k-c_1}{k}}{\binom{n+c_1-1}{k}} \frac{\beta^k}{k!} a^{(k)}(\beta)+ O_{a,b,m} \left(\frac{1}{n^{m+1}}\right) \right).
\end{align}
\end{thm}
Most proofs of Theorem \ref{darbmetthm} use contour integration, although an elementary proof was found by Knuth and Wilf \cite{knuth}. Using \eqref{eq:decompprod}, Theorem \ref{darbmetthm} gives an asymptotic expansion of $[x^n] F_G(x)$, as $n$ tends to $\infty$, in all of this paper's applications. 

Methods similar to Theorem \ref{darbmetthm} were used to study semigroup counting problems related to finite fields in the limit $n \to \infty$ and $q$ fixed, see the papers \cite{warlimont, 1mod} and the monographs \cite{knop1,knop2}.

Our objective is to establish results on the magnitude of $[x^n]F_G(x)$ in the limit $q^n \to \infty$, for which Theorem \ref{darbmetthm} is not suitable. To do so, we need to work out the dependency of the error term in \eqref{eq:darbstat} on the parameters $a$, $b$. This is done in the following theorem, under additional restrictions influenced by \eqref{eq:prop1}.
\begin{thm}\label{thm1}
Let $a(x)=\exp\left(\sum_{n \ge 1} \frac{\widetilde{a_n} x^n}{n} \right)=\sum a_n x^n$ and $b(x) = (1-\frac{x}{\beta})^{-c_1} = \sum b_n x^n$ $\left(c_1 \in \left(0,1\right) \right)$ be two power series, with radii of convergence at least $\alpha$ and exactly $\beta$, respectively. Assume that $\alpha > \beta > 0$. Assume that
\begin{equation}\label{properties}
r=\frac{\beta}{\alpha} \le \frac{1}{\sqrt{2}}.
\end{equation}
Further assume that there is a positive number $c_2$ such that
\begin{align}\label{properties3}
|\widetilde{a_n}| &\le \frac{c_2}{\alpha^n}.
\end{align}
Fix an integer $m \ge 0$. For an integer $n>m$, write the $n$th coefficient $f_n$ of $f(x)=a(x)\cdot b(x)$ as 
\begin{equation}\label{edef}
\begin{split}
[x^n] f(x) &= b_n \cdot \left( a(\beta) + \sum_{k=1}^{m} \frac{\binom{k-c_1}{k}}{\binom{n+c_1-1}{k}} \frac{\beta^k}{k!} a^{(k)}(\beta)+ E \right).
\end{split}
\end{equation}
Then
\begin{align}
\left| E\right| &\ll_m  \exp(3 c_2 r) \left( \left(\frac{r}{n} \right)^{m+1}  (c_2+m)_{m+1} + \binom{n+c_2-1}{n} \frac{4n^2}{c_1} r^n  \right) \label{eq:estimatethm1} \\
& \ll_{m,c_1,c_2}  \left(\frac{r}{n}\right)^{m+1},\label{eq:estimatethm12}
\end{align}
where $(x)_n:=x(x-1)\cdots (x-(n-1))$ is the falling factorial. For $m=0$, the implicit constant in \eqref{eq:estimatethm1} is (at most) 24.
\end{thm}
We also establish the following effective corollary.
\begin{cor}\label{nicer}
Let $a$, $b$, $\alpha$, $\beta$ be as in Theorem \ref{thm1}, and assume that they satisfy \eqref{properties} and \eqref{properties3}. Let $f(x)=a(x)\cdot b(x)$. For $n$ large enough, specifically
\begin{equation}
n \ge \max \{ 5\left(\frac{2c_2+4}{\ln(1/r)} + 1\right) \ln \left(\frac{2c_2+4}{\ln(1/r)} + 1\right)+1 ,  2\frac{\ln(c_1^{-1})}{\ln(1/r)} + 1 \},
\end{equation}
we have
\begin{equation}\label{expcor}
[x^n]f(x) = b_n\cdot \left( a(\beta)+O\left( \frac{\exp(3c_2 r) c_2 r }{n} \right)\right),
\end{equation}
where the implied constant is (at most) 48.
\end{cor}
\begin{remark}
In all of our applications, the Riemann hypothesis for function fields shows that $\beta^{-1}$ is a of the form $q^c$ for some positive integer $c$ depending on the application and $\alpha^{-1}$ is the square root of $\beta^{-1}$. Thus, $r=\beta/\alpha=q^{-c/2}$, which has two implications. First, equation \eqref{properties} is satisfied. Second, 
\begin{equation}
\lim_{q \to \infty} \frac{r}{n}=0,
\end{equation}
which by \eqref{eq:estimatethm12} shows that as $q$ tends to $\infty$, the relative error term $E$ tends to zero.
\end{remark}
\begin{remark}\label{sizeabeta}
Assumption \eqref{properties3} implies that
\begin{equation}
\begin{split}
\left| a(\beta) \right| &\ge \exp\left(\sum_{n \ge 1} \frac{-\left|\widetilde{a_n}\right| \beta^n}{n} \right)  \ge \exp\left(\sum_{n \ge 1} \frac{-c_2 r^n}{n} \right)  = (1-r)^{c_2} , \\
\left| a(\beta) \right| &\le \exp\left(\sum_{n \ge 1} \frac{\left|\widetilde{a_n}\right| \beta^n}{n} \right)  \le \exp\left(\sum_{n \ge 1} \frac{c_2 r^n}{n} \right)  = (1-r)^{-c_2}.
\end{split}
\end{equation}
In particular, we obtain the bound $a(\beta) = \exp\left( O(c_2r)\right)$.
\end{remark}
\section{Landau's theorem over \texorpdfstring{$\FF_{q}[T]$}{Fq[T]}, for \texorpdfstring{$A^2+TB^2$}{A2+TB2}}\label{land}
We begin with an overlook of the proof of Theorem \ref{landauthm}, leaving the more technical details to \S \ref{linearsubsec}. Let $G_q \subseteq \FF_q[T]$ be the set of monic polynomials of the form $A^2+TB^2$. By \cite[Thm.~2.5]{lior}, a monic polynomial $f$ is in $G_q$ if and only if the multiplicity of any monic irreducible polynomial $P$ of the form $\left(\frac{P}{T}\right)=-1$ in the factorization of $f$ is even. In other words, $G_q$ is the semigroup generated by the following set of mutually coprime polynomials:
\begin{equation}
\begin{split}
&\{ P\in \FF_q[T] \,:  P\text{ is monic irreducible satisfying } \left(\frac{P}{T} \right) = 1\} \\
& \qquad \bigcup \{ P^2 \,:  P\in \FF_q[T] \text{ is monic irreducible satisfying } \left(\frac{P}{T} \right) = -1\} \bigcup \{ T \}.
\end{split}
\end{equation}
Let $\chi_2: \FF_q[T] \to \mathbb{C}$ denote the quadratic character modulo $T$:
\begin{equation}
\chi_2(f) = \left( \frac{f}{T} \right).
\end{equation}
For any subset $S \subseteq \{ -1,0,1\}$, let $\pi_q(n;\chi_2,S)$ count the monic irreducible polynomials $P$ of degree $n$ in $\FF_q[T]$ such that $\chi_2(f)\in S$. For $s \in \{-1,0,1\}$, we write $\pi_q(n;\chi_2,s)$ for $\pi_q(n;\chi_2,\{s\})$. It follows that the generating function $F_{G_q}$ of $B(n,q)$ is
\begin{equation}\label{genfuncbnq}
\begin{split}F_{G_q}(x)&= \sum_{n\ge 0} B(n,q)x^n \\
&=\prod_{n \ge 1} (1-x^n)^{-\pi_q(n;\chi_2, \{0,1\})} (1-x^{2n})^{-\pi_q(n;\chi_2, -1)}.
\end{split}
\end{equation}
If we set 
\begin{equation}\label{gdef}
g(n) = \begin{cases} \pi_q(n;\chi_2, \{0,1\}) & 2 \nmid n, \\ \pi_q(n;\chi_2,\{0,1\}) + \pi_q(\frac{n}{2};\chi_2, -1) & 2 \mid n, \end{cases}
\end{equation}
then $F_{G_q}$ assumes the form
\begin{equation}
F_{G_q}(x) = \prod_{n \ge 1} (1-x^n)^{-g(n)}. 
\end{equation}
As before, one sees that
\begin{equation}
F_{G_q}(x) = \exp \left(\sum_{n\ge 1} \frac{\psi_{G_q}(n)x^n}{n} \right),
\end{equation}
where 
\begin{equation}
\psi_{G_q}(n)= \sum_{d \mid n} d \cdot g(d).
\end{equation}
In \S\ref{linearsubsec} we obtain a closed-form expression for $\psi_{G_q}(n)$. In particular, we find that 
\begin{equation}\label{defen}
e_n= \psi_{G_q}(n) - \frac{q^n}{2}
\end{equation}
satisfies
\begin{equation}
e_n = O(q^{\lfloor n/2 \rfloor}),
\end{equation}
and, as we saw in \eqref{eq:decompprod}, we may write $F_{G_q}$ as a product of two power series with distinct radii of convergence:
\begin{equation}\label{decompforland}
F_{G_q}(x) =  \exp\left(\sum_{n \ge 1} \frac{e_n x^n}{n} \right)\cdot (1-qx)^{-\frac{1}{2}}.
\end{equation}
Theorem \ref{thm1} with
\begin{equation}\label{chosenab}
a(x)=\exp\left(\sum_{n \ge 1} \frac{e_n x^n}{n} \right),\,  b(x)=(1-qx)^{-\frac{1}{2}}
\end{equation}
gives the asymptotics of $B(n,q)$.
\subsection{Proof of Theorem \ref{landauthm}}\label{linearsubsec}
First, we obtain an explicit formula for $e_n$.
We write the (formal) zeta function of $\mathbb{F}_{q}[T]$ and its Euler product:
\begin{equation}\label{zeta}
\begin{split}
\mathcal{Z}_{\mathbb{F}_q[T]}(u) &= \sum_{f \in \mathcal{M}_q} u^{\deg f} = (1-qu)^{-1} \\
&= \prod_{P \in \mathcal{P}_q} (1-u^{\deg P})^{-1} = \prod_{n \ge 1}(1-u^n)^{-\pi_q(n)}.
\end{split}
\end{equation}
We have $\pi_q(n) = \pi_q(n;\chi_2, \{0,1\}) + \pi_q(n;\chi_2, -1)$, and so logarithmic differentiation of \eqref{zeta} gives
\begin{equation}\label{zeta2}
q^n = \sum_{d \mid n} d \cdot \left(  \pi_q(d;\chi_2, \{0,1\}) + \pi_q(d;\chi_2, -1) \right).
\end{equation}
We consider the (formal) L-function of the quadratic character $\chi_2$:
\begin{equation}
\mathcal{L}_{\FF_q[T]}(u,\chi_2) = \sum_{f \in \mathcal{M}_q} \chi_2(f) u^{\deg f}.
\end{equation} 
On the one hand, $\mathcal{L}_{\FF_q[T]}(u,\chi_2)$ is actually the constant polynomial $1$, as half of the elements of $\FF_q^{\times}$ are squares and half are non-squares (cf. \cite[Prop.~4.3]{rosen}):
\begin{equation}\label{lfuncis1}
\mathcal{L}_{\FF_q[T]}(u,\chi_2) =1.
\end{equation}
On the other hand, since $\chi_2$ is completely multiplicative, $\mathcal{L}_{\FF_q[T]}(u,\chi_2)$ admits the following Euler product:
\begin{equation}\label{lfunceuler}
\begin{split}
\mathcal{L}_{\FF_q[T]}(u,\chi_2) &= \prod_{P\in \mathcal{P}_q:\chi_2(P)=1} (1-u^{\deg P})^{-1} \prod_{P\in \mathcal{P}_q:\chi_2(P)=-1} (1+u^{\deg P})^{-1}\\
&= \prod_{ n \ge 1} (1-u^n)^{-\pi_q(n;\chi_2, 1)} \prod_{n \ge 1} (1+u^n)^{-\pi_q(n;\chi_2, -1)}\\
&= \prod_{n \ge 1} (1-u^n)^{-\pi_q(n;\chi_2,\{0,1\})} \prod_{n \ge 1} (1+u^n)^{-\pi_q(n;\chi_2,-1)}\cdot (1-u).
\end{split}
\end{equation}
Comparing \eqref{lfuncis1} and \eqref{lfunceuler}, we have the following identity:
\begin{equation}\label{lfuncprod}
\frac{1}{1-u}  = \prod_{n \ge 1} (1-u^n)^{-\pi_q(n;\chi_2,\{0,1\})} \prod_{n \ge 1} (1+u^n)^{-\pi_q(n;\chi_2,-1)}.
\end{equation}
We take the logarithmic derivative of \eqref{lfuncprod}, and obtain the following by comparing coefficients:
\begin{equation}\label{tosimpab}
1 = \sum_{d \mid n} d \cdot \left(\pi_q(d;\chi_2,\{0,1\}) + \pi_q(d;\chi_2,-1)(-1)^{n/d}\right).
\end{equation}
Using \eqref{zeta2}, we simplify \eqref{tosimpab}:
\begin{equation}\label{tosumpowers}
\sum_{d \mid n, 2d\nmid n} d \cdot \pi_q(d;\chi_2, -1)= \frac{q^n - 1}{2}.
\end{equation}
Let $v_2(n)$ be the exponent of the greatest
power of 2 that divides $n$. To compute $\sum_{d \mid n}d \cdot \pi_q(d;\chi_2,-1)$, we plug $n,\frac{n}{2},\cdots,\frac{n}{2^{v_2(n)}}$ for $n$ in \eqref{tosumpowers} and sum:
\begin{equation}\label{bnform}
\sum_{d \mid n} d \cdot \pi_q(d;\chi_2,-1)= \sum_{i=0}^{v_2(n)} \frac{q^{n/2^i} - 1}{2}.
\end{equation}
Using \eqref{zeta2} and \eqref{bnform}, we find
\begin{equation}\label{anform}
\sum_{d \mid n} d\cdot \pi_q(d;\chi_2,\{0,1\})= \frac{q^n + 1}{2} - \sum_{i=1}^{v_2(n)} \frac{q^{n/2^i} - 1}{2}.
\end{equation}
Now we can calculate $\sum_{d \mid n} d\cdot g(d)$, using \eqref{gdef}, \eqref{bnform} and \eqref{anform}: 
\begin{equation}\label{eq:e_val}
\begin{split}
\sum_{d \mid n} d\cdot g(d) &= \sum_{d \mid n} d\cdot \pi_q(d;\chi_2,\{0,1\}) + \sum_{2\mid d \mid n} d\cdot   \pi_q(\frac{d}{2};\chi_2,-1) \\
&= \sum_{d \mid n} d\cdot \pi_q(d;\chi_2,\{0,1\}) + 2\sum_{2d' \mid n} d' \cdot \pi_q(d';\chi_2,-1)\\
&= \frac{q^n + 1}{2} - \sum_{i=1}^{v_2(n)} \frac{q^{n/2^i} - 1}{2} + 2\sum_{i=0}^{v_2(n/2)} \frac{q^{n/2^{i+1}} - 1}{2} \\
&= \frac{q^n + 1}{2} +\sum_{i=1}^{v_2(n)} \frac{q^{n/2^i} - 1}{2}.
\end{split}
\end{equation}
By \eqref{eq:e_val}, the exact value of $e_n$, as defined in \eqref{defen}, is
\begin{equation}\label{valen}
e_n = \frac{1}{2} + \sum_{i=1}^{v_2(n)} \frac{q^{n/2^i}-1}{2}.
\end{equation}
In particular, $e_n$  satisfies 
\begin{equation}\label{land_psi}
\frac{1}{2} \le e_n \le q^{\lfloor n/2 \rfloor}.
\end{equation}
We establish Theorem \ref{landauthm} with 
\begin{equation}\label{newkq}
K_q=a(q^{-1}),
\end{equation}
and show that this is the same $K_q$ defined in \eqref{eq:constantlior}. For $n=1$, Theorem \ref{landauthm} is immediate since $B(1,q)=\frac{q+1}{2}$ and, by \eqref{land_psi},
\begin{equation}\label{atag1}
\begin{split}
K_q &= a(q^{-1})= \exp\left( \sum_{n \ge 1} \frac{e_n \cdot q^{-n}}{n} \right) \\
&=\exp\left( \sum_{n \ge 1} \frac{O(q^{\lfloor n/2 \rfloor -n}) }{n} \right) = \exp \left( O\left( \frac{1}{q} \right) \right) = O(1),
\end{split}
\end{equation}
and so
\begin{equation}
B(1,q) = K_q \cdot \binom{1-\frac{1}{2}}{1} \cdot q  + O\left( 1 \right),
\end{equation}
as needed. Now assume that $n>1$. We may use Theorem \ref{thm1} with $m=1$ and the functions $a$, $b$ as defined in \eqref{chosenab}, which satisfy $a \cdot b = F_{G_q}$ as we have seen in \eqref{decompforland}. By \eqref{land_psi}, the parameters for Theorem \ref{thm1} are 
\begin{equation}\label{paramsab}
c_1=\frac{1}{2},\, c_2=1, \, \alpha = q^{-\frac{1}{2}},\, \beta=q^{-1}.
\end{equation}
We obtain the following estimate for $B(n,q)$:
\begin{equation}\label{eq:beforefinalland}
\begin{split}
\frac{B(n,q)}{\binom{n-\frac{1}{2}}{n} \cdot q^n} &= a(q^{-1}) + \frac{1}{2n-1} \frac{a'(q^{-1})}{ q } + O\left(\frac{1}{qn^2}\right).
\end{split}
\end{equation}
Theorem \ref{landauthm} follows once we show $a'(q^{-1}) = O(1)$. Note \eqref{land_psi} implies that
\begin{equation}\label{atag2}
\frac{a'(q^{-1})}{a(q^{-1})}= \sum_{n \ge 1}  e_n \cdot q^{1-n}=\sum_{n \ge 1}O(q^{1+\lfloor n/2 \rfloor -n}) = O(1).
\end{equation}
By \eqref{atag1} and \eqref{atag2}, $a'(q^{-1}) = O(1)$ is established, as needed. From \ref{eq:beforefinalland} we see that $\lim_{n\to \infty} \frac{B(n,q)}{\binom{n-\frac{1}{2}}{n}\cdot q^n} = a(q^{-1})$, while \eqref{eq:bign} implies that $\lim_{n\to \infty} \frac{B(n,q)}{\binom{n-\frac{1}{2}}{n}\cdot q^n}$ equals the constant $K_q$ defined in \eqref{eq:constantlior}. Hence $K_q$ defined in \eqref{eq:constantlior} coincides with $K_q$ defined in \eqref{newkq}, as claimed.
\begin{remark} We show that $B(n,q)$ is a polynomial in $q$ of degree $n$, and explain how to compute its first $d$ coefficients. From \eqref{genfuncbnq} and \eqref{decompforland}, we obtain
\begin{equation}\label{bnqpol}
\begin{split}
B(n,q) &= [x^n] \left( (1-qx)^{-\frac{1}{2}} \cdot \exp\left(\sum_{j \ge 1} \frac{e_j x^j}{j} \right) \right)\\
&= \sum_{i=0}^{n} q^{n-i} \binom{n-i-\frac{1}{2}}{n-i} \cdot  [x^i]\exp\left(\sum_{j \ge 1} \frac{e_j x^j}{j} \right).
\end{split}
\end{equation}
From \eqref{valen} we see that $e_n$ is a polynomial in $q$ of degree $\le \lfloor \frac{n}{2} \rfloor $. Thus, the $n$th coefficient of 
\begin{equation}
\exp\left(\sum_{j \ge 1}\frac{e_j x^j}{j}\right) = \sum_{i \ge 0} \frac{1}{i!} \left(\sum_{j \ge 1}\frac{e_j x^j}{j}\right)^i
\end{equation}
is also a polynomial in $q$ of degree at most $\lfloor \frac{n}{2} \rfloor$. Hence, for any $i \le n$,
\begin{equation}\label{nddeg}
\deg_q \left( q^{n-i} \binom{n-i-\frac{1}{2}}{n-i} \cdot  [x^i]\exp\left(\sum_{j \ge 1} \frac{e_j x^j}{j} \right) \right) \le n-i + \lfloor \frac{i}{2} \rfloor =n-\lceil \frac{i}{2} \rceil,
\end{equation}
and if $i=0$ then equality holds in \eqref{nddeg}. Hence, \eqref{nddeg} implies that $B(n,q)$ is a polynomial in $q$ of degree $n$. If one only wants the first $d$ coefficients of $B(n,q)$, then only the first $2d-1$ coefficients of $\exp\left(\sum_{j \ge 1} \frac{e_j x^j}{j} \right)$ are needed, as may be seen from the following identity, which follows from \eqref{bnqpol} and \eqref{nddeg}:
\begin{equation}
B(n,q) = \sum_{i=0}^{2d-2} q^{n-i} \binom{n-i-\frac{1}{2}}{n-i} \cdot  [x^i]\exp\left(\sum_{j \ge 1} \frac{e_j x^j}{j} \right) + O_n(q^{n-d}).
\end{equation}
\end{remark}
\subsection{The constant \texorpdfstring{$K_q$}{Kq}}\label{kqsec}
We study the constant appearing in the main term of Theorem \ref{landauthm}, $K_q$. In  \eqref{eq:constantlior}, the constant is expressed as the following Euler product:
\begin{equation}
K_q=(1-q^{-1})^{-\frac{1}{2}}\prod_{P \in \mathcal{P}_q:(\frac{P}{T})=-1}(1-q^{-2\deg P})^{-\frac{1}{2}}. 
\end{equation}
We provide another expression, which is the polynomial analogue of the following similar expression obtained in the integer setting independently in \cite{shanks, flajvardi}:
\begin{equation}\label{beautifulflaj}
K = \frac{1}{\sqrt{2}} \prod_{n \ge 1} \left( \left(1-\frac{1}{2^{2^n}}\right)\frac{\zeta(2^n)}{L(2^n,\chi)}\right)^{1/2^{n+1}},
\end{equation}
where $\zeta(s)=\sum_{n \ge 1} n^{-s}$,  $L(s,\chi)=\sum_{n\ge 1} \frac{\chi(n)}{n^s}$ and where $\chi$ is the principal character modulo 4.

Although the method of derivation of \eqref{beautifulflaj} applies \textit{mutatis mutandis} to the polynomial setting, we do it slightly differently. In the process we obtain a functional equation for $F_{G_q}(x)$.
\begin{lem}\label{lempower2}
Let $A(x) = \exp\left(\sum_{n\ge 1} \frac{a_n x^n}{n} \right)$ be a formal power series. Let 
\begin{equation*}
b_n = \begin{cases} a_n-a_{n/2} & 2 \mid n \\ a_n & 2 \nmid n \end{cases} \text{ and }B(x) = \exp\left(\sum_{n \ge 1} \frac{b_n x^n}{n} \right). 
\end{equation*}
Then 
\begin{equation}\label{poweridentity}
A(x) = \prod_{k \ge 0} B(x^{2^k})^{2^{-k}},
\end{equation}
where each of the roots $B(x^{2^k})^{2^{-k}}$ is chosen so that the constant term is $1$.

In particular,
\begin{equation}\label{secondpartiden}
\frac{A^2(x)}{A(x^2)} = B^2(x).
\end{equation}
\end{lem}
\begin{proof}
Assertion \eqref{secondpartiden} follows from \eqref{poweridentity}. We have 
\begin{equation}
\begin{split}
\prod_{k \ge 0} B(x^{2^k})^{2^{-k}} &= \prod_{k \ge 0} \exp\left(\sum_{n \ge 1} \frac{b_n x^{2^k n}}{2^k n} \right)\\
& = \exp\left( \sum_{n \ge 1} \frac{ \left(\sum_{i=0}^{v_2(n)} b_{n/2^i} \right) x^n}{n} \right),
\end{split}
\end{equation}
so it suffices to show
\begin{equation}
a_n = \sum_{i=0}^{v_2(n)} b_{\frac{n}{2^i}},
\end{equation}
which follows from the definition of $b_n$.
\end{proof}
Applying Lemma \ref{lempower2} to $A = F_{G_q}$, we find
\begin{align}
\frac{F_{G_q}^2(x)}{F_{G_q}(x^2)} &= h^2(x),
\end{align}
where
\begin{equation}
h(x)=\exp\left(\sum_{n \ge 1} \left(\psi_{G_q}(n)-\underbrace{\psi_{G_q}\left(\frac{n}{2}\right)}_{=0\text{ if }2\nmid n}\right) \frac{x^n}{n}\right).
\end{equation}
Using the formula $\psi_{G_q}(n) = \frac{q^n+1}{2} +\sum_{i=1}^{v_2(n)} \frac{q^{n/2^i}-1}{2}$ (see \eqref{valen}) we find 
\begin{equation}
\psi_{G_q}(n)-\underbrace{\psi_{G_q}\left(\frac{n}{2}\right)}_{=0\text{ if }2\nmid n} = \frac{q^n-(-1)^n}{2},
\end{equation}
thus
\begin{equation}
h(x) = \sqrt{\frac{1+x}{1-qx}},
\end{equation}
and 
\begin{align}
\frac{F_{G_q}^2(x)}{F_{G_q}(x^2)} &= \frac{1+x}{1-qx}.
\end{align}
This functional equation determines $F_{G_q}(x)$ uniquely, assuming $F_{G_q}$ is a power series with $F_{G_q}(0)\neq 0$.

Applying Lemma \ref{lempower2} to $A(x)=a(x)=\exp\left(\sum_{n\ge 1} \frac{e_n x^n}{n} \right)$, we find
\begin{equation}\label{prod2}
\exp\left(\sum_{n\ge 1} \frac{e_n x^n}{n} \right) = \prod_{k \ge 0} \widetilde{h}(x^{2^k})^{2^{-k}},
\end{equation}
where 
\begin{equation}
\widetilde{h}(x)=\frac{\sqrt{1+x}}{\sqrt[4]{1-qx^2}}.
\end{equation}
Plugging $x=q^{-1}$ in \eqref{prod2} gives the following identity:
\begin{equation}\label{identitykq}
K_q = \prod_{k \ge 0} \frac{\sqrt[2^{k+1}]{1+q^{-2^k}}}{\sqrt[2^{k+2}]{1-q^{1-2^{k+1}}}} = \frac{\sqrt{1+q^{-1}}}{\sqrt[4]{1-q^{-1}}} \frac{\sqrt[4]{1+q^{-2}}}{\sqrt[8]{1-q^{-3}}} \frac{\sqrt[8]{1+q^{-4}}}{\sqrt[16]{1-q^{-7}}} \cdots,
\end{equation}
which shows that $K_q$ is an analytic function of $q^{-1}$.

Next we discuss the the analogy between \eqref{identitykq} and \eqref{beautifulflaj}. Define $|f|=q^{\deg f}$, and put 
\begin{equation}\label{zetas}
\begin{split}
\zeta_{\FF_q[T]}(s)&=\sum_{f \in \mathcal{M}_q} |f|^{-s} =\frac{1}{1-q^{1-s}},\\
L_{\FF_q[T]}(s,\chi_2)&=\sum_{f \in \mathcal{M}_q} \chi_2(f) |f|^{-s} = 1.
\end{split}
\end{equation}
One may rearrange the terms of the product in \eqref{identitykq} in a way that gives an expression which is analogous to \eqref{beautifulflaj}:
\begin{align}
K_q &= \frac{1}{\sqrt{1-q^{-1}}} \prod_{n \ge 1} \left( \left(1-\frac{1}{q^{2^n}}\right)\frac{\zeta_{\FF_q[T]}(2^n)}{L_{\FF_q[T]}(2^n,\chi_2)}\right)^{1/2^{n+1}}.
\end{align}
\subsection{Second-order term for fixed \texorpdfstring{$q$}{q}}
In the integer setting, Shanks \cite{shanks} computed numerically the coefficient $c$ appearing in the following expansion:
\begin{equation}
\sum_{n\le x} b(n) = K \frac{x}{\sqrt{\ln x}} \left(1 + \frac{c}{\ln x} + O\left(\frac{1}{\ln^2 x}\right)\right),
\end{equation}
and found that $c \approx 0.581948659$. We study a similar problem in the polynomial setting.

Equation \eqref{eq:beforefinalland} shows that in the limit $n \to \infty$, we may expand $B(n,q)$ as follows:
\begin{equation}
B(n,q) = K_q \cdot \binom{n-\frac{1}{2}}{n}  \cdot q^n \left( 1 + \frac{a'(q^{-1})}{2q  \cdot a(q^{-1})} \frac{1}{n} +  O\left(\frac{1}{qn^2}\right) \right).\end{equation}
We give a formula for the constant 
\begin{equation}
c_{q} =\frac{a'(q^{-1})}{2q  \cdot a(q^{-1})}=  \frac{1}{2}\sum_{i \ge 1} e_i q^{-i},
\end{equation}
using the exact formula \eqref{valen} for $e_n$:
\begin{equation}\label{estcq}
\begin{split}
c_q &=\frac{1}{2} \sum_{i \ge 1} e_i q^{-i} = \frac{1}{4}  \sum_{i \ge 1}\left( q^{-i} + \sum_{k=1}^{v_2(i)} \left(q^{\frac{i}{2^k}-i}-q^{-i}\right) \right)\\
&=\frac{1}{4} \left( \frac{1}{q-1} + \sum_{j \ge 1} \sum_{i' \ge 1} q^{i'(1-2^j)}-q^{-i'2^j} \right)\\
&= \frac{1}{4} \left( \frac{1}{q-1} + \sum_{j \ge 1} \left( \frac{1}{q^{2^j-1}-1} - \frac{1}{q^{2^j}-1} \right) \right)\\
\end{split}
\end{equation}
In particular, \eqref{estcq} gives the following estimate for $c_q$:
\begin{equation}
c_q= \frac{1}{2q} + O\left(\frac{1}{q^2}\right).
\end{equation}
\section{Landau's theorem over \texorpdfstring{$\FF_{q}[T]$}{Fq[T]}, for \texorpdfstring{$A^2-\alpha B^2$}{A2+B2}}\label{landab}
\subsection{Proof of Theorem \ref{thmchinintimprovquad}}
We only prove estimate \eqref{eq:b1est}, as \eqref{eq:b2est}, \eqref{eq:b3est} are proven \textit{mutatis mutandis}.

Let $q$ be a prime power. As explained in the beginning of \S\ref{furthapp}, given a non-square element $\alpha \in \FF_q$, the set $S_1(q)$ defined in \eqref{s1def} is the set of monic polynomials of the form $A^2-\alpha B^2$. By definition, $S_1(q)$ is the multiplicative semigroup of $\FF_q[T]$ generated by the following set of mutually coprime polynomials:
\begin{equation}
\begin{split}
&\{ P\in \FF_q[T] \,:  P\text{ is monic irreducible of even degree}\} \\
& \qquad \bigcup \{ P^2 \,:  P\in \FF_q[T] \text{ is monic irreducible of odd degree}\}.
\end{split}
\end{equation}
It follows that the generating function of $B_1(2n,q)$ (counting function of $S_1(q)$, defined in \eqref{b1def}) is
\begin{equation}\label{gens1}
\begin{split}F_{S_1(q)}(x)&= \sum_{n\ge 0} B_1(2n,q) x^n \\
&= \prod_{P\in \mathcal{P}_q:\,2 \mid \deg P} (1-x^{\deg P /2})^{-1} \prod_{P\in \mathcal{P}_q:\,2 \nmid \deg P} (1-x^{\deg P})^{-1} \\
&=\prod_{n \ge 1} (1-x^n)^{-\pi_q(2n)} (1-x^{2n-1})^{-\pi_q(2n-1)}.
\end{split}
\end{equation}
If we set 
\begin{equation}\label{hdef}
h(n) = \begin{cases} \pi_q(2n) + \pi_q(n) & 2 \nmid n, \\ \pi_q(2n) & 2 \mid n, \end{cases}
\end{equation}
then $F_{S_1(q)}$ assumes the form
\begin{equation}
F_{S_1(q)}(x) = \prod_{n \ge 1} (1-x^n)^{-h(n)}. 
\end{equation}
As before, one sees that
\begin{equation}
F_{S_1(q)}(x) = \exp \left(\sum_{n\ge 1} \frac{\psi_{S_1(q)}(n)x^n}{n} \right),
\end{equation}
where 
\begin{equation}\label{psis1def}
\psi_{S_1(q)}(n)= \sum_{d \mid n} d \cdot h(d).
\end{equation}
We evaluate $\psi_{S_1(q)}(n)$. Let $v_2(n)$ be the exponent of the greatest power of 2 that divides $n$. By \eqref{psis1def}, we have
\begin{equation}\label{psifeval}
\begin{split}
\psi_{S_1(q)}(n) &= \sum_{d \mid n} d \cdot\pi_q(2d) + \sum_{d \mid n,\, 2\nmid d} d \cdot \pi_d(d) \\
&= \frac{1}{2} \left( \sum_{d' \mid 2n} d' \cdot  \pi_q(d') - \sum_{d' \mid 2n,\, 2 \nmid d'} d' \cdot  \pi_q(d')\right) +\sum_{d \mid n,\, 2 \nmid d} d \cdot\pi_q(d) \\
&= \frac{1}{2} \left( \sum_{d' \mid 2n} d' \cdot  \pi_q(d') \right)+ \frac{1}{2}\left( \sum_{d' \mid n, \, 2\nmid d'} d' \cdot  \pi_q(d') \right)\\
&= \frac{1}{2} \left( \sum_{d' \mid 2n} d' \cdot  \pi_q(d') \right)+ \frac{1}{2}\left( \sum_{d' \mid \frac{n}{2^{v_2(n)}}} d' \cdot  \pi_q(d') \right).
\end{split}
\end{equation}
By \eqref{classicpsi} and \eqref{psifeval}, we find
\begin{equation}
\psi_{S_1(q)}(n) = \frac{q^{2n}}{2} + \frac{q^{\left(\frac{n}{2^{v_2(n)}}\right)}}{2}.
\end{equation}
If we define
\begin{equation}\label{deffn}
f_n= \psi_{S_1(q)}(n) - \frac{q^{2n}}{2},
\end{equation}
we see by \eqref{psifeval} that
\begin{equation}\label{valfn}
f_n = \frac{q^{\left( \frac{n}{2^{v_2(n)}}\right)}}{2} = O(q^n).
\end{equation}
As we saw in \eqref{eq:decompprod}, we may write $F_{S_1(q)}$ as a product of two power series with distinct radii of convergence:
\begin{equation}\label{decompforlandclass}
F_{S_1(q)}(x) =  \exp\left(\sum_{n \ge 1} \frac{f_n x^n}{n} \right)\cdot (1-q^2x)^{-\frac{1}{2}}.
\end{equation}
We may apply Theorem \ref{thm1} with $m=0$ and
\begin{equation}\label{chosenabclass}
a(x)=\exp\left(\sum_{n \ge 1} \frac{f_n x^n}{n} \right),\,  b(x)=(1-q^2x)^{-\frac{1}{2}}.
\end{equation}
By \eqref{valfn}, the parameters for Theorem \ref{thm1} are
\begin{equation}\label{paramsabclassicland}
c_1=\frac{1}{2},\, c_2=\frac{1}{2}, \, \alpha = q^{-1},\, \beta=q^{-2}.
\end{equation}
By Theorem \ref{thm1} we find
\begin{equation}
B_1(2n,q) = \binom{n-\frac{1}{2}}{n} \cdot q^{2n} \cdot \left(a(q^{-2}) +  O\left( \frac{1}{qn} \right) \right).
\end{equation}
This establishes \eqref{eq:b1est} with \begin{equation}\label{cq1def}
C_{q,1} = a(q^{-2}).
\end{equation}
By Remark \ref{sizeabeta}, 
\begin{equation}
C_{q,1} = 1+O\left( \frac{1}{q} \right),
\end{equation}
which concludes the proof.
\subsection{The constant \texorpdfstring{$C_{q,1}$}{Cq1}}\label{cq1sec}
We study the constant appearing in the main term of \eqref{eq:b1est}, $C_{q,1}$. First, we show how we may express it as a product over primes. By \eqref{decompforlandclass} \eqref{chosenabclass}, and \eqref{paramsabclassicland},
\begin{equation}\label{justmove}
a(x) = \frac{F_{S_1(q)}(x)}{b(x)}
\end{equation}
for any $|x| < q^{-2}$. Letting $x \to \left( q^{-2} \right)^{-}$ in \eqref{justmove} and using \eqref{cq1def}, we obtain
\begin{equation}\label{cq1form}
C_{q,1} = \lim_{x \to \left(q^{-2}\right)^{-}}\frac{F_{S_1(q)}(x)}{b(x)}.
\end{equation}
For $0<x<q^{-2}$, $b$ has the following Euler product:
\begin{equation}\label{beuler}
\begin{split}
b(x) &= (1-q^{2}x)^{-\frac{1}{2}} = (1-q\sqrt{x})^{-\frac{1}{2}} (1+q\sqrt{x})^{-\frac{1}{2}} \\
&= \mathcal{Z}_{\FF_q[T]}^{1/2}(\sqrt{x}) \cdot \mathcal{Z}_{\FF_q[T]}^{1/2}(-\sqrt{x}) \\
&= \prod_{P \in \mathcal{P}_q} (1-x^{\deg P/2})^{-1/2} \cdot \prod_{P \in \mathcal{P}_q} (1-(-1)^{\deg {P}}x^{\deg P/2})^{-1/2},
\end{split}
\end{equation}
where $\mathcal{Z}_{\FF_q[T]}(x)$ is defined in \eqref{zeta}. By considering odd-degree and even-degree primes separately, equation \eqref{beuler} shows
\begin{equation}\label{beuler2}
b(x) = \prod_{P \in \mathcal{P}_q: \, 2\mid \deg P} (1-x^{\deg P /2})^{-1}\prod_{P \in \mathcal{P}_q: \, 2\nmid \deg P} (1-x^{\deg P})^{-1/2}.
\end{equation} 
By \eqref{cq1form}, \eqref{beuler2} and the Euler product for $F_{S_1(q)}$ given in \eqref{gens1}, we obtain the following expression for $C_{q,1}$:
\begin{equation}
\begin{split}
C_{q,1} &= \lim_{x \to \left(q^{-2} \right)^{-}} \frac{\prod_{P\in \mathcal{P}_q:\,2 \mid \deg P} (1-x^{\deg P /2})^{-1} \prod_{P\in \mathcal{P}_q:\,2 \nmid \deg P} (1-x^{\deg P})^{-1}  }{\prod_{P \in \mathcal{P}_q: \, 2\mid \deg P} (1-x^{\deg P/2})^{-1} \prod_{P \in \mathcal{P}_q: \, 2\nmid \deg P} (1-x^{\deg P})^{-1/2}} \\
&= \lim_{x \to \left(q^{-2} \right)^{-}} \prod_{P\in \mathcal{P}_q : \, 2 \nmid \deg P} (1-x^{\deg P})^{-\frac{1}{2}}\\
&= \prod_{P\in \mathcal{P}_q : \, 2 \nmid \deg P} (1-q^{-2\deg P})^{-\frac{1}{2}}.
\end{split}
\end{equation}
We provide another expression for $C_{q,1}$. Applying Lemma \ref{lempower2} to $A(x)=\exp\left(\sum_{n\ge 1} \frac{f_n x^n}{n} \right)$, we find
\begin{equation}\label{prod2cq1}
\exp\left(\sum_{n\ge 1} \frac{f_n x^n}{n} \right) = \prod_{k \ge 0} B(x^{2^k})^{2^{-k}},
\end{equation}
where 
\begin{equation}
B(x) = \exp\left( \sum_{2 \nmid n} \frac{q^n x^n}{2n} \right) = \frac{\exp\left( \sum_{n \ge 1} \frac{q^n x^n}{2n} \right)}{\exp\left( \sum_{2 \mid n} \frac{q^n x^n}{2n} \right)} = \frac{(1-qx)^{-\frac{1}{2}}}{(1-q^2x^2)^{-\frac{1}{4}}} = \left( \frac{1+qx}{1-qx} \right)^{\frac{1}{4}}.
\end{equation}
Plugging $x=q^{-2}$ in \eqref{prod2cq1} and using \eqref{cq1def}, we obtain
\begin{equation}\label{identitycq1}
C_{q,1} = \prod_{k \ge 0} \left(\frac{1+q^{1-2^{k+1}}}{1-q^{1-2^{k+1}}}\right)^{-2^{-k-2}}.
\end{equation}
Next we discuss the the analogy between \eqref{identitycq1} and \eqref{beautifulflaj}. Let $\eta: \FF_q[T] \to \mathbb{C}$ be the (generalized) Dirichlet character $\eta(f) = (-1)^{\deg (f)}$, and define $|f|=q^{\deg f}$. Let $\zeta_{\FF_q[T]}$ be as in \eqref{zetas} and 
\begin{equation}
\begin{split}
L_{\FF_q[T]}(s,\eta)&=\sum_{f \in \mathcal{M}_q} \eta(f) |f|^{-s},
\end{split}
\end{equation}
which equals
\begin{equation}
L_{\FF_q[T]}(s,\eta) =\sum_{n \ge 0} (-1)^n q^{n} q^{-ns} = \frac{1}{1+q^{1-s}}.
\end{equation}
The identity \eqref{identitycq1} may be written differently as an expression which is analogous to \eqref{beautifulflaj}:
\begin{align}\label{eq:cq1alt}
C_{q,1} &= \prod_{n \ge 1} \left( \frac{\zeta_{\FF_q[T]}(2^n)}{L_{\FF_q[T]}(2^n,\eta)}\right)^{1/2^{n+1}}.
\end{align}
\subsection{Second-order term for fixed \texorpdfstring{$q$}{q}}
Consider the functions $a$, $b$ defined in \eqref{chosenabclass}. By \eqref{decompforlandclass}, the product $a \cdot b$ is the generating function of $B_1(2n,q)$. Hence, if we apply Theorem \ref{thm1} with $m=1$ to the functions $a$, $b$, we find that in the limit $n \to \infty$, we may expand $B_1(2n,q)$ as follows:
\begin{equation}
B_1(2n,q) = C_{q,1} \cdot \binom{n-\frac{1}{2}}{n}  \cdot q^{2n} \left( 1 + \frac{a'(q^{-2})}{2q^2  \cdot a(q^{-2})} \frac{1}{n} +  O\left(\frac{1}{q^2n^2}\right) \right).
\end{equation}
We give a formula for the constant $c'_{q} =\frac{a'(q^{-2})}{2q^2  \cdot a(q^{-2})}=  \frac{1}{2}\sum_{i \ge 1} f_i q^{-2i}$ using the exact formula \eqref{deffn} for $f_n$:
\begin{equation}\label{estcq1}
\begin{split}
c'_q =\frac{1}{2} \sum_{i \ge 1} f_i q^{-2i} =\frac{1}{4}  \sum_{i \ge 1} q^{\left(\frac{i}{2^{v_2(i)}}\right)-2i}.
\end{split}
\end{equation}
In particular, equation \eqref{estcq1} gives the following estimate for $c'_q$:
\begin{equation}
c'_q = \frac{1}{4q} + O\left( \frac{1}{q^3}\right).
\end{equation}

\section{Counting divisors over function fields}\label{divfunc}
Here we prove Theorem \ref{thmchinintimprov}. First, we prove an auxiliary lemma.
\subsection{Counting primes}
\begin{lem}\label{divlem}
Let $K$ be a function field with finite constant field $\FF_q$, i.e. $K$ is a finitely generated field extension of transcendence degree one over $\FF_q$ and $\FF_q$ is algebraically closed in $K$.
Fix a prime $\mathfrak{Q}$ of $K$.
For any $n\ge 1$, let 
\begin{equation}\label{pikdef}
\pi_{K,\mathfrak{Q}}(n) = \# \{ \text{Primes of degree }n \text{ in }K \}  \setminus \{ \mathfrak{Q} \}.
\end{equation}
Given integers $r \ge 2$, $\ell \ge 1$, define the following functions:
\begin{align}
\psi_{r,K}(n)&= \sum_{d \mid n} d\cdot \pi_{K,\mathfrak{Q}}(rd), \label{eq:psirkdef}\\
f_{r,\ell, K}(n)&=\pi_{K,\mathfrak{Q}}(rn) - \underbrace{\pi_{K,\mathfrak{Q}}(\frac{rn}{\ell+1})}_{=0\text{ if }\ell+1 \nmid rn},\label{eq:frellkdef}\\
\psi_{r,\ell,K}(n)&= \sum_{d \mid n} d \cdot f_{r,\ell,K}(d), \label{eq:psirellkdef}\\
\xi_{r,K}(n) &= \sum_{d \mid nr} \frac{d \cdot \pi_{K,\mathfrak{Q}}(d)}{(d,r)} \label{eq:xidef}.
\end{align}
Denote by $g_K$ the genus of $K$ and let
\begin{equation}
M_{K,\mathfrak{Q}} = \max\{ g_K, \deg \mathfrak{Q} \}.
\end{equation}
Then
\begin{align}
\psi_{r,K}(n) &= \frac{q^{rn}}{r} + O\left( \frac{M_{K,\mathfrak{Q}}}{r} q^{rn/2}\right),\label{eq:finalval}\\
\psi_{r,\ell,K}(n) &= \frac{q^{rn}}{r} + O\left( \frac{M_{K,\mathfrak{Q}}}{r} q^{rn/2}\right),\label{eq:finalval2}\\
\xi_{r,K}(n) &= \frac{q^{rn}}{r} + O\left( \frac{M_{K,\mathfrak{Q}}}{r} q^{rn/2}\right),\label{eq:finalval3}
\end{align}
where the implied constants are (at most) 16, 42 and 50, respectively.
\end{lem}
\begin{proof}
Let
\begin{equation}
\pi_{K} = \# \{ \text{Primes of degree }n \text{ in }K \}.
\end{equation}
The zeta function of $K$ is 
\begin{equation}\label{zetadef}
\zeta_{K}(u) =\sum_{D \in \text{Div}_{\ge 0}(K)} u^{\deg D} = \prod_{ n \ge 1} (1-u^n)^{-\pi_K(n)} = \prod_{ n \ge 1} (1-u^n)^{-\pi_{K,\mathfrak{Q}}(n)} \cdot (1-u^{\deg \mathfrak{Q}})^{-1}.
\end{equation}
The function $\zeta_{K}(u)$ is a rational function in $u$ of the form 
\begin{equation}\label{zetadecomp}
\zeta_K(u) = \frac{L_K(u)}{(1-u)(1-qu)},
\end{equation}
where $L_K$ is a polynomial with $L_K(0)=1$ and of degree twice the genus of $K$, $2g_K$. The Riemann hypothesis for $\zeta_K$, proved by Weil, shows that the absolute value of the roots of $L_K$ is $q^{-1/2}$ \cite[Thm.~A.7]{rosen}. Hence, taking the logarithmic derivative of \eqref{zetadef}, \eqref{zetadecomp} and equating coefficients, we find
\begin{equation}\label{riemannestk}
\sum_{d \mid n} d \cdot \pi_{K, \mathfrak{Q}}(d) = q^n + 1 - \deg \mathfrak{Q} \cdot 1_{\deg \mathfrak{Q} \mid n}+O(g_K q^{n/2}) = q^n + O\left( M_{K,\mathfrak{Q}} q^{n/2} \right),
\end{equation}
with implied constant (at most) 3. We estimate $\psi_{r,K}(n)$ as follows, using \eqref{riemannestk}:
\begin{equation}\label{eq:divpsi}
\begin{split}
\sum_{d \mid n} d \cdot \pi_{K,\mathfrak{Q}}(rd) &=  \frac{1}{r} \sum_{d \mid n} rd \cdot \pi_{K,\mathfrak{Q}}(rd)  \\
&= \frac{1}{r} \left(\sum_{d' \mid rn} d' \cdot \pi_{K,\mathfrak{Q}}(d')- \sum_{r \nmid d' \mid rn} d' \cdot \pi_{K,\mathfrak{Q}}(d')\right)\\
&= \frac{1}{r} \left(q^{rn} +O( M_{K,\mathfrak{Q}} q^{rn/2})- \sum_{r \nmid d' \mid rn} d' \cdot \pi_{K,\mathfrak{Q}}(d')\right),
\end{split}
\end{equation}
where the implied constant is $3$.
If $d'$ is an integer dividing $nr$ but not divisible by $r$, it means $d' \mid \frac{nr}{p}$ for some prime $p$ dividing $nr$. Hence, using \eqref{riemannestk}, we see that the sum $\sum_{r \nmid d' \mid nr} d' \cdot \pi_{K,\mathfrak{Q}}(d')$ is at most
\begin{equation}\label{eq:taildivisor}
\begin{split}
\sum_{r \nmid d' \mid nr} d' \cdot \pi_{K,\mathfrak{Q}}(d') &\le \sum_{p \mid nr} \sum_{d' \mid \frac{nr}{p}} d' \cdot \pi_{K,\mathfrak{Q}}(d') \\
&\le \sum_{p \mid nr} \left( q^{nr/p} +3 M_{K,\mathfrak{Q}} q^{\frac{nr}{2p}} \right) \\
&\le q^{\lfloor \frac{nr}{2} \rfloor} \left( \sum_{i \ge 0} q^{-i} \right) +3  M_{K,\mathfrak{Q}} q^{\lfloor \frac{nr}{2} \rfloor /2} \left( \sum_{i \ge 0} q^{-\frac{i}{2}} \right) \\
&\le \frac{1}{1-q^{-1}} q^{\lfloor \frac{nr}{2} \rfloor} +3  M_{K,\mathfrak{Q}} \frac{1}{1-q^{-1/2} }q^{\lfloor \frac{nr}{2} \rfloor} \\
& \le (2+ \frac{3 M_{K,\mathfrak{Q} }}{1-2^{-1/2}} )q^{\frac{nr}{2}} \le 13  M_{K,\mathfrak{Q}} q^{\frac{nr}{2}}.
\end{split}
\end{equation}
Combining \eqref{eq:divpsi} and  \eqref{eq:taildivisor}, we obtain \eqref{eq:finalval} with an implied constant $3+13=16$. We now prove \eqref{eq:finalval2}. Note that \eqref{riemannestk} implies
\begin{equation}\label{crudepik}
\pi_{K,\mathfrak{Q}}(n) \le \frac{4}{n} M_{K,\mathfrak{Q}} q^{n}.
\end{equation}
If $\ell +1 \nmid rn$, then $\psi_{r,\ell,K}(n) = \psi_{r,K}(n)$, and \eqref{eq:finalval2} is established. Otherwise, by \eqref{crudepik}, 
\begin{equation}\label{givingval2}
\begin{split}
\sum_{d \mid n,\, \ell+1 \mid rd} d \cdot \pi_{K,\mathfrak{Q}}\left(\frac{rd}{\ell+1}\right) &\le 4 M_{K,\mathfrak{Q}}\frac{\ell+1}{r} \sum_{d \mid n,\, \ell+1 \mid rd} q^{\frac{rd}{\ell+1}} \\
&\le 4 M_{K,\mathfrak{Q}} \frac{\ell+1}{r} q^{\frac{rn}{\ell+1}} \sum_{i \ge 0}q^{-i} \\
&\le 8 M_{K,\mathfrak{Q}} \frac{\ell+1}{r} q^{\frac{rn}{\ell+1}}.
\end{split}
\end{equation}
If $\ell=1$, then \eqref{givingval2} gives 
\begin{equation}\label{ell1}
\sum_{d \mid n,\, \ell+1 \mid rd} d \cdot \pi_{K,\mathfrak{Q}}\left(\frac{rd}{\ell+1}\right) \le 16\frac{M_{K,\mathfrak{Q}}}{r} q^{\frac{rn}{2}}.
\end{equation}
By \eqref{ell1} and \eqref{eq:finalval}, estimate \eqref{eq:finalval2} is established with the absolute constant 16+16=32.

If $\ell >1$, then \eqref{givingval2} gives
\begin{equation}\label{givingval22}
\sum_{d \mid n,\, \ell+1 \mid rd} d \cdot \pi_{K,\mathfrak{Q}}\left(\frac{rd}{\ell+1}\right)  \le \frac{8 M_{K,\mathfrak{Q}} }{r} \cdot rn \cdot q^{\frac{rn}{3}}.
\end{equation}
Note that for any positive $x$.
\begin{equation}\label{anyxexp}
x = \frac{6}{e\ln 2} \cdot e \cdot \frac{\ln 2}{6} x \le \frac{6}{e\ln 2}  \cdot e \cdot \exp(  (\frac{\ln 2}{6} x) - 1) = \frac{6}{e \ln 2}\cdot 2^{x/6} \le   \frac{6}{e \ln 2} \cdot  q^{x/6}.
\end{equation} 
Using \eqref{anyxexp} with $x=rn$ in \eqref{givingval22}, we obtain
\begin{equation}
\sum_{d \mid n,\, \ell+1 \mid rd} d \cdot \pi_{K,\mathfrak{Q}}\left(\frac{rd}{\ell+1}\right) \le \frac{8 \cdot  \frac{6}{e \ln 2} M_{K,\mathfrak{Q}} }{r} q^{\frac{rn}{2}}.
\end{equation}
By \eqref{ell1} and \eqref{eq:finalval}, estimate \eqref{eq:finalval2} is established with the absolute constant $16+8\cdot  \frac{6}{e \ln 2} \le 42$. We now prove \eqref{eq:finalval3}. From \eqref{eq:psirkdef} and \eqref{eq:xidef}, we have
\begin{equation}\label{diffpsixi}
\begin{split}
\left| \xi_{r,K}(n) - \psi_{r,K}(n) \right| & \le \sum_{d \mid nr, d \neq nr} \frac{d}{(d,r)} \cdot \pi_{K,\mathfrak{Q}}(d) \\
& \le \sum_{d \le \frac{nr}{3}} d \cdot \pi_{K,\mathfrak{Q}}(d) + 1_{2 \mid rn} \cdot \frac{rn/2}{(rn/2,r)} \cdot \pi_{K,\mathfrak{Q}}(rn/2).
\end{split}
\end{equation}
From \eqref{crudepik} and \eqref{anyxexp}, we obtain
\begin{equation}\label{firstpartxi}
\begin{split}
\sum_{d \le \frac{nr}{3}} d \cdot \pi_{K,\mathfrak{Q}}(d) &\le 4 M_{K,\mathfrak{Q}} \sum_{d \le \frac{nr}{3}} q^d \le 8M_{K,\mathfrak{Q}} q^{rn/3} \\
& \le 8M_{K,\mathfrak{Q}} \frac{6}{e \ln 2}\frac{1}{rn} q^{rn/2}  \le  \frac{26 M_{K,\mathfrak{Q}}}{r} q^{rn/2}.
\end{split}
\end{equation}
From \eqref{crudepik}, we also have 
\begin{equation}\label{secondpartxi}
1_{2 \mid rn} \cdot \frac{rn/2}{(rn/2,r)} \cdot \pi_{K,\mathfrak{Q}}(rn/2) \le n \cdot \frac{8}{rn} M_{K,\mathfrak{Q}} q^{rn/2} =  \frac{8M_{K,\mathfrak{Q}}}{r} q^{rn/2}.
\end{equation}
From \eqref{diffpsixi}, \eqref{firstpartxi}, \eqref{secondpartxi} and \eqref{eq:finalval}, we arrive at estimate \eqref{eq:finalval3} with the absolute implied constant $16+26+8=50$. 
\end{proof}
\subsection{Proof of Theorem \ref{thmchinintimprov}}\label{gen3func}
We first treat the case $\deg_{K} \mathfrak{Q} \nmid rn$. The divisors counted by $B_i(rn,rk)$ and $B_3(rn,r,\ell,K)$ satisfy
\begin{equation}\label{conddeg}
\deg_{\mathfrak{Q}} (D) = rn.
\end{equation}
By \eqref{degqdef}, the condition \eqref{conddeg} is equivalent to
\begin{equation}\label{newconddeg}
v_{\mathfrak{Q}}  = -\frac{rn}{\deg_{K} \mathfrak{Q}}.
\end{equation}
The right-hand side of \eqref{newconddeg} is not integral when $\deg_{K} \mathfrak{Q} \nmid rn$, hence 
\begin{equation}
B_1(rn,r,K)=B_2(rn,r,K)=B_3(rn,r,\ell,K)=0.
\end{equation}
From now on we assume 
\begin{equation}
\deg_{K} \mathfrak{Q} \mid rn.
\end{equation}
Given a divisor $D = \sum_{ \mathfrak{P} \in \mathbb{P}_K} v_{\mathfrak{P}}(D) \cdot \mathfrak{P} \in \text{Div}_{0,\mathfrak{Q}}(K)$, we define
\begin{equation}
D_{0} = \sum_{ \mathfrak{Q} \neq \mathfrak{P} \in \mathbb{P}_K} v_{\mathfrak{P}}(D) \cdot \mathfrak{P}, \quad \quad D_{\infty} = -v_{\mathfrak{Q}}(D) \cdot \mathfrak{Q}.
\end{equation}
Note that $D = D_{0} - D_{\infty}$. A divisor $D \in \text{Div}_{0,\mathfrak{Q}}(K)$ is counted by $B_1(rn,r,K) / B_2(rn,r,K) / B_3(rn,r,\ell,K)$ if and only if
\begin{equation}
D_{\infty} = \frac{rn}{\deg_{K} \mathfrak{Q}} \cdot \mathfrak{Q}
\end{equation}
and 
\begin{equation}\label{d0cond}
D_{0} \in  \left\{ \sum_{\mathfrak{Q} \neq \mathfrak{P}  \in \mathbb{P}_K} n_{\mathfrak{P}} \cdot \mathfrak{P}, n_{\mathfrak{P}} \ge 0  : n_{\mathfrak{P}} > 0 \implies \begin{cases} \frac{r}{(r,\deg_{K} \mathfrak{P})} \mid n_{\mathfrak{P}} & (B_1(rn,r,K)) \\ r \mid \deg_{K} \mathfrak{P} & (B_2(rn,r,K)) \\ r \mid \deg_{K} \mathfrak{P} \text{ and }n_{\mathfrak{P}} \le \ell & (B_3(rn,r,\ell,K)) \end{cases}\right\}.
\end{equation}
Note that condition \eqref{d0cond} does not depend on $\deg_K \mathfrak{Q}$, so we may assume $\deg_{K} \mathfrak{Q} = 1$ without loss of generality. The generating function of $\{B_1(rn,r,K)\}_{n \ge 0}$ is, by definition,
\begin{equation}
\begin{split}
G_{r,K}(x) &= \prod_{\mathfrak{P} \in \mathbb{P}_K \setminus \{ \mathfrak{Q} \}} (1-x^{\frac{\deg_K \mathfrak{P}}{(r,\deg_K \mathfrak{P})}})^{-1} \\
 & =\prod_{d \mid r} \prod_{ \substack{\mathfrak{P} \in \mathbb{P}_K \setminus \{ \mathfrak{Q} \} \\ (r,\deg_{K} \mathfrak{P}) = d}} (1-x^{\frac{\deg_K \mathfrak{P}}{d}})^{-1}\\
  & = \exp\left( \sum_{d \mid r} \sum_{m \ge 1} \frac{x^m}{m} \frac{\sum_{\substack{\mathfrak{P} \in \mathbb{P}_K \setminus \{ \mathfrak{Q} \} \\ (\deg_{K} \mathfrak{P},r) = d \\ \deg_{K} \mathfrak{P} \mid md }} \deg_{K} \mathfrak{P}}{d} \right) \\
 & = \exp\left( \sum_{m \ge 1} \frac{\xi_{r,K}(m) x^m}{m} \right).
\end{split}
\end{equation}
The generating function $F_{r,K}$ of $\{ B_2(rm,r,K)\}_{m \ge 0}$ is, by definition, 
\begin{equation}
\begin{split}
F_{r,K}(x)&=\prod_{m \ge 1}(1+x^m+x^{2m} +\cdots )^{\pi_{K,\mathfrak{Q}}(rm)}\\
&=\prod_{m \ge 1}(1-x^m)^{-\pi_{K,\mathfrak{Q}}(rm)}\\
&=\exp\left(\sum_{m \ge 1} \frac{ \psi_{r,K}(m) x^m}{m} \right),
\end{split}
\end{equation}
where $\pi_{K,\mathfrak{Q}}$ is defined in \eqref{pikdef} and $\psi_{r,K}$ is defined in \eqref{eq:psirkdef}. The generating function $F_{r,\ell,K}$ of $\{ B_3(rm,r,\ell,K) \}_{n\ge 0}$ is, by definition,
\begin{equation}\label{eq:deffrl}
\begin{split}
F_{r,\ell, K}(x)&=\prod_{m\ge 1} (1+x^{m}+\cdots + x^{m\ell})^{\pi_{K,\mathfrak{Q}}(rm)} \\
&=\prod_{m\ge 1} \left(\frac{1-x^{m(\ell + 1)}}{1-x^m}\right)^{\pi_{K,\mathfrak{Q}}(rm)}.
\end{split}
\end{equation}
We may rearrange \eqref{eq:deffrl} as follows:
\begin{equation}
F_{r,\ell,K}(x)= \prod_{m \ge 1} (1-x^m)^{-f_{r,\ell,K}(m)} = \exp\left(\sum_{m \ge 1} \frac{\psi_{r,\ell,K}(m)x^m}{m} \right),
\end{equation}
where $f_{r,\ell,K}$ and $\psi_{r,\ell,K}$ are defined in \eqref{eq:frellkdef} and \eqref{eq:psirellkdef}, respectively. 
The estimates \eqref{eq:finalval}, \eqref{eq:finalval2} and \eqref{eq:finalval3} of Lemma \ref{divlem} show that we may write $G_{r,K}$, $F_{r,K}$ and $F_{r,\ell,K}$ as
\begin{equation}\label{defa123}
\begin{split}
G_{r, K} &= a_1 \cdot b,\\
F_{r,K} &= a_2 \cdot b, \\
F_{r, \ell, K} &= a_3 \cdot b,
\end{split}
\end{equation}
where 
\begin{equation}\label{defa1a2}
\begin{split}
b(x) &= \exp\left( \sum_{m \ge 1} \frac{q^{rm} x^m}{m} \right)= (1-q^{r}x)^{-\frac{1}{r}},\\
a_1 &= \frac{G_{r,K}}{b} = \exp \left( \sum_{m\ge  1} \frac{\widetilde{a_{1,m}} x^m}{m} \right),\\
a_2 &= \frac{F_{r,K}}{b} = \exp \left( \sum_{m\ge  1} \frac{\widetilde{a_{2,m}} x^m}{m} \right),\\
a_3&= \frac{F_{r,\ell, K}}{b} = \exp \left( \sum_{m\ge  1} \frac{\widetilde{a_{3,m}} x^m}{m} \right),
\end{split}
\end{equation}
\begin{equation}\label{otermindiv}
\begin{split}
i=1,2,3: \left| \widetilde{a_{i,m}}\right| &= O\left( \frac{M_{K,\mathfrak{Q}}}{r} q^{rm/2} \right),
\end{split}
\end{equation}
where the implied constant in \eqref{otermindiv} is (at most) 50.

We apply Corollary \ref{nicer} to the pairs $(a_1,b)$, $(a_2,b)$ and $(a_3,b)$ with the following parameters, given by \eqref{defa1a2} and \eqref{otermindiv}:
\begin{equation}
c_1=\frac{1}{r},\, c_2 \le  \frac{50 M_{K,\mathfrak{Q}}}{r}, \, \alpha = q^{-r/2}, \, \beta=q^{-r}.
\end{equation}
We obtain that the estimates  \eqref{eq:b11est}, \eqref{eq:b22est} and  \eqref{eq:b33est} of Theorem \ref{thmchinintimprov} hold with an absolute implied constant, as long as 
\begin{equation}\label{uglyrange}
n \ge \max\{ 5\cdot \left( \frac{\frac{4\cdot 50 \cdot M_{K,\mathfrak{Q}}}{r}+8}{r \ln q}+1 \right) \ln  \left( \frac{\frac{4\cdot 50 \cdot M_{K,\mathfrak{Q}}}{r}+8}{r \ln q}+1 \right)+ 1, \frac{4}{\ln q} \frac{\ln r}{r } +1\},
\end{equation}
and the constants $C_{1,r,K}, C_{2,r,K}, C_{r,\ell,K}$ are given by
\begin{equation}\label{defc1c2c3}
C_{1,r,K} = a_1(q^{-r}),\quad C_{2,r,K} = a_2(q^{-r}),\quad C_{r,\ell,K} = a_3(q^{-r}).
\end{equation}
By Remark \ref{sizeabeta}, 
\begin{equation}
C_{1,r,K},C_{2,r,K},C_{r,\ell,K} = \exp\left( O\left( \frac{M_{K,\mathfrak{Q}}}{rq^{r/2}} \right) \right) = 1+O_{M_{K,\mathfrak{Q}}}\left( \frac{1}{rq^{r/2}} \right).
\end{equation}
Since
\begin{equation*}
\frac{4}{\ln q} \frac{\ln r}{r } +1  =O(1) \text{ and } \frac{\frac{4\cdot 50 \cdot M_{K,\mathfrak{Q}}}{r}+8}{r \ln q}+1  =O\left(\frac{M_{K,\mathfrak{Q}}}{r^2 \ln q} + 1\right), 
\end{equation*}
we get that the range \eqref{uglyrange} may be replaced with the range \eqref{rangeofexp} of Theorem \ref{thmchinintimprov}, as needed.
\subsection{The constants \texorpdfstring{$C_{1,r,K}, C_{2,r,K}, C_{r,\ell,K}$}{C1rK, C2rK, CrlK}}\label{secthreeconst}
Here we give an expression for the constants $C_{1,r,K}$, $C_{2,r,K}$ and $C_{r,\ell,K}$, involving a product over primes of $K$. The Euler products for the generating functions $G_{r,K},F_{r,K},F_{r,\ell,K}$ defined in \S\ref{gen3func} are given by
\begin{equation}\label{euler3func}
\begin{split}
G_{r,K}(x) &= \prod_{\mathfrak{Q} \neq \mathfrak{P} \in \mathbb{P}_K,} (1-x^{\frac{\deg_{K} \mathfrak{P}}{(r,\deg_{K} \mathfrak{P})}})^{-1},\\
F_{r,K}(x) &= \prod_{\mathfrak{Q} \neq \mathfrak{P} \in \mathbb{P}_K, r \mid \deg_{K} \mathfrak{P}} (1-x^{\deg_{K} \mathfrak{P}/r})^{-1}, \\
F_{r,\ell,K}(x) &= \prod_{\mathfrak{Q} \neq \mathfrak{P} \in \mathbb{P}_K, r \mid \deg_{K} \mathfrak{P}} \frac{1-x^{\deg_{K} \mathfrak{P}(\ell+1)/r}}{1-x^{\deg_{K} \mathfrak{P}/r}}.
\end{split}
\end{equation}
Let $K_r$ be the constant field extension of $K$ of degree $r$, and let 
\begin{equation}
\mathcal{Z}_{K_r}(u) = \prod_{\mathfrak{P} \in \mathbb{P}_{K_r}} (1-u^{\deg_{K_r} \mathfrak{P}})^{-1}
\end{equation}
be its zeta function. When $r=1$, $\mathcal{Z}_{K_r} = \mathcal{Z}_{K}$ is just the zeta function of $K$. By \cite[Lem.~8.14 and Thm.~8.15]{rosen},
\begin{equation}\label{coolzetar}
\mathcal{Z}_{K_r}(u^r) = \prod_{\omega:\omega^r = 1}\mathcal{Z}_{K}(\omega u) = \prod_{\mathfrak{P} \in \mathbb{P}_K} (1-u^{\frac{r \deg_K \mathfrak{P}}{(r,\deg_{K} \mathfrak{P})}})^{-(r,\deg_{K} \mathfrak{P})}.
\end{equation}
We separate the product in \eqref{coolzetar} into a product over primes of degree divisible by $r$ and the rest, and replace $u^r$ with $x$:
\begin{equation}\label{zkrprod}
\mathcal{Z}_{K_r}(x) = \prod_{\mathfrak{P} \in \mathbb{P}_K, r \mid \deg_{K} \mathfrak{P}} (1-x^{\frac{\deg_K \mathfrak{P}}{r}})^{-r}  \prod_{\mathfrak{P} \in \mathbb{P}_K, r \nmid \deg_{K} \mathfrak{P}} (1-x^{\frac{ \deg_K \mathfrak{P}}{(r,\deg_{K} \mathfrak{P})}})^{-(r,\deg_{K} \mathfrak{P})}.
\end{equation}
Let $\mathcal{Z}_{K_r}(x)^{1/r}$ be chosen so that the constant term is $1$. By \eqref{zkrprod},
\begin{equation}\label{zkrprodroot}
\mathcal{Z}_{K_r}(x)^{1/r} = \prod_{\mathfrak{P} \in \mathbb{P}_K, r \mid \deg_{K} \mathfrak{P}} (1-x^{\frac{\deg_K \mathfrak{P}}{r}})^{-1}  \prod_{\mathfrak{P} \in \mathbb{P}_K, r \nmid \deg_{K} \mathfrak{P}} (1-x^{\frac{ \deg_K \mathfrak{P}}{(r,\deg_{K} \mathfrak{P})}})^{-(r,\deg_{K} \mathfrak{P})/r}.
\end{equation}
Dividing the expressions in \eqref{euler3func} by \eqref{zkrprodroot}, we obtain
\begin{equation}\label{ratio3func}
\begin{split}
\frac{G_{r,K}(x)}{\mathcal{Z}_{K_r}(x)^{1/r}} &= (1-x^{\frac{\deg_{K} \mathfrak{Q}}{(r,\deg_{K} \mathfrak{Q})}}) \cdot \prod_{\mathfrak{P} \in \mathbb{P}_K, r \nmid \deg_{K} \mathfrak{P}} (1-x^{\frac{ \deg_K \mathfrak{P}}{(r,\deg_{K} \mathfrak{P})}})^{(r,\deg_{K} \mathfrak{P})/r-1}, \\
\frac{F_{r,K}(x)}{\mathcal{Z}_{K_r}(x)^{1/r}} &= \prod_{\mathfrak{P} \in \mathbb{P}_K, r \nmid \deg_{K} \mathfrak{P}} (1-x^{\frac{ \deg_K \mathfrak{P}}{(r,\deg_{K} \mathfrak{P})}})^{(r,\deg_{K} \mathfrak{P})/r} \cdot \begin{cases} 1 & r \nmid \deg_{K} \mathfrak{Q}, \\  (1-x^{\deg_{K} \mathfrak{Q}/r}) & r \mid \deg_{K} \mathfrak{Q}, \end{cases}\\
\frac{F_{r,\ell,K}(x)}{\mathcal{Z}_{K_r}(x)^{1/r}} &=\prod_{\mathfrak{P} \in \mathbb{P}_K, r \nmid \deg_{K} \mathfrak{P}} (1-x^{\frac{ \deg_K \mathfrak{P}}{(r,\deg_{K} \mathfrak{P})}})^{(r,\deg_{K} \mathfrak{P})/r} \cdot \begin{cases} 1 & r \nmid \deg_{K} \mathfrak{Q}, \\  (1-x^{\deg_{K} \mathfrak{Q}/r}) & r \mid \deg_{K} \mathfrak{Q} \end{cases} \\
& \quad \cdot \prod_{\mathfrak{Q} \neq \mathfrak{P} \in \mathbb{P}_K, r \mid \deg_{K} \mathfrak{P}} (1-x^{\deg_{K} \mathfrak{P}(\ell+1)/r}).
\end{split}
\end{equation}
From \cite[Thm.~5.9]{rosen} we know that $\mathcal{Z}_{K_r}(x)$ is a rational function of $x$ of the form
\begin{equation}\label{formofzkr}
\mathcal{Z}_{K_r}(x)=\frac{L_{K_r}(x)}{(1-q^r x)(1-x)},
\end{equation}
where $L_{K_r}$ is a polynomial of degree $2g_{K_r}= 2g_K$ satisfying 
\begin{equation}
L_{K_r}(u) = L_{K_r}\left(\frac{1}{q^r u}\right) (q^ru^2)^{g_K}
\end{equation}
and 
\begin{equation}\label{lkrval1}
L_{K_r}(q^{-r})=q^{-r \cdot g_K} L_{K_r}(1) = q^{-r \cdot g_K} h_{K_r},
\end{equation}
where $h_{K_r}$ is the class number of $K_r$.
Combining \eqref{ratio3func} and \eqref{formofzkr}, we obtain
\begin{equation}
\begin{split}
G_{r,K}(x) &= (1-q^r x)^{-1/r} \left(\frac{L_{K_r}(x)}{1-x}\right)^{1/r} \cdot (1-x^{\frac{\deg_{K} \mathfrak{Q}}{(r,\deg_{K} \mathfrak{Q})}}) \cdot \prod_{\mathfrak{P} \in \mathbb{P}_K, r \nmid \deg_{K} \mathfrak{P}} (1-x^{\frac{ \deg_K \mathfrak{P}}{(r,\deg_{K} \mathfrak{P})}})^{(r,\deg_{K} \mathfrak{P})/r-1} \\
F_{r,K}(x) &= (1-q^r x)^{-1/r} \left(\frac{L_{K_r}(x)}{1-x}\right)^{1/r} \cdot \prod_{\mathfrak{P} \in \mathbb{P}_K, r \nmid \deg_{K} \mathfrak{P}} (1-x^{\frac{ \deg_K \mathfrak{P}}{(r,\deg_{K} \mathfrak{P})}})^{(r,\deg_{K} \mathfrak{P})/r} \cdot \begin{cases} 1 & r \nmid \deg_{K} \mathfrak{Q}, \\  (1-x^{\deg_{K} \mathfrak{Q}/r}) & r \mid \deg_{K} \mathfrak{Q}, \end{cases}\\
F_{r,\ell,K}(x) &= (1-q^r x)^{-1/r} \left(\frac{L_{K_r}(x)}{1-x}\right)^{1/r} \cdot \prod_{\mathfrak{P} \in \mathbb{P}_K, r \nmid \deg_{K} \mathfrak{P}} (1-x^{\frac{ \deg_K \mathfrak{P}}{(r,\deg_{K} \mathfrak{P})}})^{(r,\deg_{K} \mathfrak{P})/r} \cdot \begin{cases} 1 & r \nmid \deg_{K} \mathfrak{Q}, \\  (1-x^{\deg_{K} \mathfrak{Q}/r}) & r \mid \deg_{K} \mathfrak{Q} \end{cases} \\
& \quad \cdot \prod_{\mathfrak{Q} \neq \mathfrak{P} \in \mathbb{P}_K, r \mid \deg_{K} \mathfrak{P}} (1-x^{\deg_{K} \mathfrak{P}(\ell+1)/r}).
\end{split}
\end{equation}
We deduce that the functions $a_1,a_2,a_3$ defined in \eqref{defa123} are given by
\begin{equation}\label{a123exp}
\begin{split}
a_1(x) &= \left(\frac{L_{K_r}(x)}{1-x}\right)^{1/r} \cdot (1-x^{\frac{\deg_{K} \mathfrak{Q}}{(r,\deg_{K} \mathfrak{Q})}}) \cdot \prod_{\mathfrak{P} \in \mathbb{P}_K, r \nmid \deg_{K} \mathfrak{P}} (1-x^{\frac{ \deg_K \mathfrak{P}}{(r,\deg_{K} \mathfrak{P})}})^{(r,\deg_{K} \mathfrak{P})/r-1},\\
a_2(x) &= \left(\frac{L_{K_r}(x)}{1-x}\right)^{1/r}  \prod_{\mathfrak{P} \in \mathbb{P}_K, r \nmid \deg_{K} \mathfrak{P}} (1-x^{\frac{ \deg_K \mathfrak{P}}{(r,\deg_{K} \mathfrak{P})}})^{(r,\deg_{K} \mathfrak{P})/r} \cdot \begin{cases} 1 & r \nmid \deg_{K} \mathfrak{Q}, \\  (1-x^{\deg_{K} \mathfrak{Q}/r}) & r \mid \deg_{K} \mathfrak{Q}, \end{cases}\\
a_3(x) &= \left(\frac{L_{K_r}(x)}{1-x}\right)^{1/r} \cdot \prod_{\mathfrak{P} \in \mathbb{P}_K, r \nmid \deg_{K} \mathfrak{P}} (1-x^{\frac{ \deg_K \mathfrak{P}}{(r,\deg_{K} \mathfrak{P})}})^{(r,\deg_{K} \mathfrak{P})/r} \cdot \begin{cases} 1 & r \nmid \deg_{K} \mathfrak{Q}, \\  (1-x^{\deg_{K} \mathfrak{Q}/r}) & r \mid \deg_{K} \mathfrak{Q} \end{cases} \\
& \quad \cdot \prod_{\mathfrak{Q} \neq \mathfrak{P} \in \mathbb{P}_K, r \mid \deg_{K} \mathfrak{P}} (1-x^{\deg_{K} \mathfrak{P}(\ell+1)/r}).
\end{split}
\end{equation}
From \eqref{a123exp} and \eqref{lkrval1}, we get that the constants $C_{1,r,K}, C_{2,r,K}, C_{r,\ell,K}$ defined in \eqref{defc1c2c3} are given by
\begin{equation}
\begin{split}
C_{1,r,K} &= a_1(q^{-r}) = q^{-g_K} \sqrt[r]{\frac{h_K}{1-q^{-r}}} \cdot (1-q^{-\frac{r \cdot\deg_{K} \mathfrak{Q}}{(r,\deg_{K} \mathfrak{Q})}}) \cdot \prod_{\mathfrak{P} \in \mathbb{P}_K, r \nmid \deg_{K} \mathfrak{P}} (1-q^{-\frac{r \cdot \deg_K \mathfrak{P}}{(r,\deg_{K} \mathfrak{P})}})^{\frac{(r,\deg_{K} \mathfrak{P})}{r}-1},\\
C_{2,r,K} &= a_2(q^{-r}) = q^{-g_K} \sqrt[r]{\frac{h_K}{1-q^{-r}}} \cdot \prod_{\mathfrak{P} \in \mathbb{P}_K, r \nmid \deg_{K} \mathfrak{P}} (1-q^{-\frac{r \cdot \deg_K \mathfrak{P}}{(r,\deg_{K} \mathfrak{P})}})^{\frac{(r,\deg_{K} \mathfrak{P})}{r}} \cdot \begin{cases} 1 & r \nmid \deg_{K} \mathfrak{Q}, \\  (1-q^{-\deg_{K} \mathfrak{Q}}) & r \mid \deg_{K} \mathfrak{Q}, \end{cases}\\
C_{r,\ell,K}&=a_3(q^{-r}) = q^{-g_K} \sqrt[r]{\frac{h_K}{1-q^{-r}}} \cdot \prod_{\mathfrak{P} \in \mathbb{P}_K, r \nmid \deg_{K} \mathfrak{P}} (1-q^{-\frac{r \cdot \deg_K \mathfrak{P}}{(r,\deg_{K} \mathfrak{P})}})^{\frac{(r,\deg_{K} \mathfrak{P})}{r}} \cdot \begin{cases} 1 & r \nmid \deg_{K} \mathfrak{Q}, \\  (1-q^{-\deg_{K} \mathfrak{Q}}) & r \mid \deg_{K} \mathfrak{Q} \end{cases} \\
& \quad \cdot \prod_{\mathfrak{Q} \neq \mathfrak{P} \in \mathbb{P}_K, r \mid \deg_{K} \mathfrak{P}} (1-q^{-(\ell + 1) \cdot \deg_{K} \mathfrak{P}}).
\end{split}
\end{equation}
\section{Polynomials with prime factors in an arithmetic progression}\label{arithprogsec}
Let $a,m \in \FF_q[T]$ be a pair of coprime polynomials with $\deg m>0$, $m$ monic. Let $G_{a,m} \subseteq \FF_q[T]$ be the set of monic polynomials whose monic prime factors lie in  the arithmetic progression $a(T) \bmod m(T)$. Let $\pi_q(n;a,m)$ count the number of monic primes of degree $n$ in the arithmetic progression $a(T) \bmod m(T)$. The generating function of $S(n;a,m)$ (counting function of $G_{a,m}$, defined in \eqref{snamdef}) is 
\begin{equation}
F_{G_{a,m}}(x) = \prod_{n \ge 1}(1-x^n)^{-\pi_q(n;a,m)}.
\end{equation}
As before, one sees that
\begin{equation}
F_{G_{a,m}}(x)= \exp \left(\sum_{n \ge 1} \frac{\psi_{G_{a,m}}(n)x^n}{n}\right),
\end{equation}
where
\begin{equation}
\psi_{G_{a,m}}(n) = \sum_{d \mid n} d \cdot \pi_q(d;a,m).
\end{equation}
We estimate $\psi_{G_{a,m}}(n)$. By a result of Wan \cite[Thm.~5.1]{wan},
\begin{align}\label{eq:wanineq}
\left| n \cdot \pi_q(n;a,m) - \frac{q^n}{\phi(m)} \right| &\le (\deg (m)+1)q^{\frac{n}{2}}.
\end{align}
Hence, 
\begin{align}\label{eq:usewanined}
\left| \psi_{G_{a,m}}(n) - \frac{q^n}{\phi(m)} \right| &\le (\deg (m)+1))q^{\frac{n}{2}} + \sum_{d \mid n, d<n} d \cdot \pi_q(d;a,m).
\end{align}
The tail $\sum_{d \mid n, d<n} d \cdot \pi_q(d;a,m)$ is easily bounded:
\begin{equation}\label{eq:tailofarith}
\begin{split}
\sum_{d \mid n, d<n} d \cdot \pi_q(d;a,m) &\le \sum_{d \mid n, d<n} d \cdot \pi_q (d)  \le \sum_{d \mid n, d<n} q^d \\
&\le q^{\frac{n}{2}} \frac{1}{1-q^{-1}} \le 2q^{\frac{n}{2}}.
\end{split}
\end{equation}
By \eqref{eq:usewanined} and \eqref{eq:tailofarith},
\begin{equation}
\left| \psi_{G_{a,m}}(n)  - \frac{q^n}{\phi(m)} \right| \le (\deg(m) +3 ) q^{\frac{n}{2}}.
\end{equation}
Hence we may write 
\begin{equation}
F_{G_{a,m}}(x) = a(x) \cdot b(x),
\end{equation}
where 
\begin{equation}
\begin{split}
a(x) &= \exp\left(\sum_{n \ge 1}\frac{\widetilde{a_n} x^n}{n} \right),\\
\left| \widetilde{a_n} \right| & \le (\deg(m)+3) q^{\frac{n}{2}},\\
b(x) &= (1-qx)^{-\frac{1}{\phi(m)}}.
\end{split}
\end{equation}
The conditions of Corollary \ref{nicer} then hold with the following parameters:
\begin{equation}\label{paramsarith}
\alpha = \frac{1}{\sqrt{q}}, \beta= \frac{1}{q}, c_1 = \frac{1}{\phi(m)} \ge q^{-\deg m},  c_2 = \deg (m) +3,
\end{equation}
and so estimate \eqref{snamest} holds with an implied constant at most 48, whenever
\begin{equation}\label{annyrange}
n \ge \max\{5\left(\frac{2(2c_2+4)}{\ln(q)}+1 \right)\ln \left(\frac{2(2c_2+4)}{\ln(q)}+1 \right) + 1, 4 \frac{\ln \left(c_1^{-1}\right)}{\ln q}+1 \}.
\end{equation}
By plugging the parameters \eqref{paramsarith} in \eqref{annyrange}, 
and replacing $\ln(q)$ with the lower bound $\ln(2)$, the range \eqref{annyrange} may be replaced with the smaller range
\begin{equation}
n \ge \max\{ 5\left( \frac{4 \deg (m) + 20 + \ln(2)}{\ln (2)} \right) \ln \left( \frac{4 \deg (m) + 20 + \ln(2)}{\ln (2)} \right)+1, 4\deg (m) + 1\}.
\end{equation}
Replacing $\frac{4 \deg (m) + 20 + \ln(2)}{\ln (2)}$ with the upper bound $6\deg (m) + 30$, Theorem \ref{semiariththm} follows.
\section{Asymptotics}\label{asympsection}
Here we prove Theorem \ref{thm1}, a theorem in analysis on which we relied in previous sections, and whose proof implies Corollary \ref{nicer}.
We break the proof into several auxiliary lemmas.
\subsection{Basic identities and inequalities}
\begin{lem}\label{lem1}
For any $c_1 \notin \mathbb{Z}$ and integers $n \ge i \ge 0$, the following identity holds:
\begin{equation}\label{magicidet}
(-1)^i\frac{\binom{-c_1}{n-i}}{\binom{-c_1}{n}} = \sum_{k=0}^{i} \frac{\binom{i}{k}\binom{k-c_1}{k}}{\binom{n+c_1-1}{k}}.
\end{equation}
\end{lem}
\begin{proof}
Define the forward difference operator $\Delta(f)(x) := f(x+1)-f(x)$, acting on $\mathbb{C}[x]$. For $f(x) = \binom{x}{n}$ we have $\Delta(f)(x) =\binom{x}{n-1}$ (Pascal's identity $\binom{x+1}{n} - \binom{x}{n} = \binom{x}{n-1}$) and so by induction we arrive at
\begin{equation}\label{comp1diff}
\Delta(f)^{(i)}(x) = \binom{x}{n-i}.
\end{equation} 
On the other hand, we have in general
\begin{equation}\label{comp2diff}
\Delta(f)^{(i)}(x) = \sum_{k=0}^{i}\binom{i}{k}(-1)^{i-k}f(x+k).
\end{equation}
Plugging $x=-c_1$ in \eqref{comp1diff}, \eqref{comp2diff} and comparing the results, we obtain
\begin{equation} \label{diffeq}
\binom{-c_1}{n-i} = \sum_{k=0}^{i} \binom{i}{k} (-1)^{i-k} \binom{k-c_1}{n}.
\end{equation}
Dividing both sides of \eqref{diffeq} by $(-1)^i \binom{-c_1}{n}$, we obtain
\begin{equation}
(-1)^i\frac{\binom{-c_1}{n-i}}{\binom{-c_1}{n}} = \sum_{k=0}^{i} \binom{i}{k}(-1)^k\frac{\binom{k-c_1}{n}}{\binom{-c_1}{n}}.
\end{equation}
The proof concludes by observing \begin{align}
(-1)^k\frac{\binom{k-c_1}{n}}{\binom{-c_1}{n}} &= \frac{\binom{k-c_1}{k}}{\binom{n+c_1-1}{k}}.
\end{align}
\end{proof}
Before the next lemma we recall that $(x)_n$ stands for $x(x-1)\cdots (x-(n-1))$, the falling factorial.
\begin{lem}\label{lem2}
The following bound holds for any $m<i \le n$ and $c_1 \in (0,1)$:
\begin{equation}\label{diffbylem1}
\sum_{k=m+1}^{i} \frac{\binom{i}{k}\binom{k-c_1}{k}}{\binom{n+c_1-1}{k}} \le \begin{cases} \frac{(i)_{m+1}}{(n+c_1-1)_{m+1}} (i-m)  & i<n, \\ n\left( \ln (n-m)+ \frac{2}{c_1}  \right) & i=n. \end{cases}
\end{equation}
\end{lem}
\begin{proof}
Since $\binom{k-c_1}{k} = \prod_{i=1}^{k} (1-\frac{c_1}{i}) \le 1$, 
we can bound the left-hand side of \eqref{diffbylem1} from above by
\begin{align}\label{eq:sum3cases}
\sum_{k=m+1}^{i} \frac{\binom{i}{k} \binom{k-c_1}{k}}{\binom{n+c_1-1}{k}} &\le  \sum_{k=m+1}^{i} \frac{(i)_k}{(n+c_1-1)_k} \\
& = \frac{(i)_{m+1}}{(n+c_1-1)_{m+1}} \sum_{k=m+1}^{i} \frac{(i-m-1)_{k-m-1}}{(n+c_1-m-2)_{k-m-1}}. \label{eq:sum3cases2}
\end{align}
The case $i<n$ follows by observing that each of the terms in the sum in \eqref{eq:sum3cases2} is at most $1$.

For the case $i=n$, note that the right-hand side of \eqref{eq:sum3cases} is equal to
\begin{equation}
\label{in}\sum_{k=m+1}^{n} \frac{\binom{n}{k}}{\binom{n+c_1-1}{k}}.
\end{equation}
The term $k=n$ in \eqref{in} gives $\frac{\binom{n}{k}}{\binom{n+c_1-1}{k}} = \frac{n}{c_1} \frac{1}{\binom{n+c_1-1}{n-1}} \le \frac{n}{c_1}$. The sum of the rest of the terms is bounded from above by
\begin{equation}
\sum_{k=m+1}^{n-1} \frac{\binom{n}{k}}{\binom{n-1}{k}} = \sum_{k=m+1}^{n-1} \frac{n}{n-k} =  n \sum_{i=1}^{n-m-1} \frac{1}{i} \le  n \left(\ln (n-m) + 1 \right) \le n\left(\ln (n-m) + \frac{1}{c_1} \right),
\end{equation}
as needed.
\end{proof}
\begin{lem}\label{lem3} Let $a(x)=\exp\left(\sum_{n \ge 1} \frac{\widetilde{a_n} x^n}{n} \right)=\sum a_n x^n$ be a power series satisfying \eqref{properties3} for some $\alpha, c_2 >0$. We have the following for any $i \ge 0$ and any $0 \le x < \alpha$:
\begin{equation}
\begin{split}
\left| a^{(i)}(x) \right|  &\le \alpha^{-i} (c_2+i-1)_i (1-\frac{x}{\alpha})^{-c_2-i}.
\end{split}
\end{equation}
\end{lem}
\begin{proof}
Property \eqref{properties3} allows us to bound $|a^{(i)}(x)|$ from above by $|G^{(i)}(x)|$ for $0 \le x < \alpha$, where 
\begin{equation}
G(x) = \exp\left(\sum_{n \ge 1} \frac{c_2 \alpha^{-n} x^n}{n}\right) = \exp\left( -c_2 \log (1-\frac{x}{\alpha})\right) = (1-\frac{x}{\alpha})^{-c_2}.
\end{equation}
By repeated differentiation and using $(-c_2)_i = (c_2+i-1)_i (-1)^i$, we find 
\begin{equation}
|G^{(i)}(x)| = \alpha^{-i} (c_2+i-1)_i (1-\frac{x}{\alpha})^{-c_2-i},
\end{equation}
as needed.
\end{proof}
\begin{lem}\label{lemmaforcor}
Let $t >0$. If $n$ is a positive integer satisfying 
\begin{equation}\label{condineq}
n \ge 1+5\left(t + 1\right) \ln \left( t+ 1\right),
\end{equation}
then the inequality 
\begin{equation}\label{ineqnt}
\frac{n-1}{\ln(n)+1} \ge t
\end{equation}
holds.
\end{lem}
\begin{proof}
Let $f(x)=\frac{x-1}{\ln(x)+1}$. Note $f$ is monotone increasing on $[1,\infty)$:
\begin{equation}
f'(x) = \frac{\ln(x)+\frac{1}{x}}{(\ln (x)+1)^2} > 0.
\end{equation}
Thus, to prove \eqref{ineqnt}, it is enough to show that
\begin{equation}\label{ineqnt2}
f(1+5(t+1)\ln(t+1)) \ge t,
\end{equation}
that is,
\begin{equation}\label{lnineq}
\frac{5(t+1)\ln(t+1)}{\ln(1+5(t+1)\ln(t+1))+1} \ge t.
\end{equation}
Let
\begin{equation}
g(t)=(t+1)\ln(t+1).
\end{equation}
Note that $g(t)-t$ is monotone increasing for $t>0$, since
\begin{equation}\label{mon}
\text{for all } t > 0: \, (g(t)-t)' = \ln(t+1) > 0.
\end{equation}
In particular, \eqref{mon} implies
\begin{equation}\label{gined}
\text{for all } t \ge 0: \,g(t) \ge t + g(0)=t
\end{equation}
and $g(t)$ is monotone increasing in $[0,\infty)$.
We split the proof of \eqref{lnineq} into two cases. If $t \le 4$, then by \eqref{gined} we obtain
\begin{equation}
\frac{5(t+1)\ln(t+1)}{\ln(1+5(t+1)\ln(t+1))+1} \ge\frac{5t}{\ln(1+5(4+1)\ln(4+1))+1} \ge t.
\end{equation}
If $t \ge 4$, then $5(t+1)\ln(t+1) \ge 1$. In particular, using $\ln(t+1) \le t < t+1$, we obtain the following when $t \ge 4$:
\begin{equation}
\begin{split}
\frac{5(t+1)\ln(t+1)}{\ln(1+5(t+1)\ln(t+1))+1} &\ge \frac{5t\ln(t+1)}{\ln(10(t+1)\ln(t+1))+1} \\
&\ge \frac{5t\ln(t+1)}{\ln(10(t+1)^2)+1} = 5t \left( \frac{1}{2} - \frac{\ln(10) + 1}{4\ln(t+1) + 2\ln(10)+2} \right)  \\
&\ge 5t \left( \frac{1}{2} - \frac{\ln(10) + 1}{4\ln(4+1) + 2\ln(10)+2} \right) \ge t,
\end{split}
\end{equation}
as needed.
\end{proof}
\subsection{Main sum inequality}
\begin{lem}\label{sumlem}
Let $a(x)=\exp\left(\sum_{n \ge 1} \frac{\widetilde{a_n} x^n}{n} \right)=\sum a_n x^n$ be a power series satisfying \eqref{properties3} for some $\alpha, c_2 > 0$. Let $\beta>0$ be a real number such that $r=\beta/\alpha$ satisfies \eqref{properties}. Let $0\le m<n$ and define
\begin{equation}\label{s2def}
S_1 = \sum_{i=m+1}^{n} \beta^{i} \left( \sum_{k=m+1}^{i} \frac{\binom{i}{k}\binom{k-c_1}{k}}{\binom{n+c_1-1}{k}} \right) \frac{a^{(i)}(0)}{i!}.
\end{equation}
Then
\begin{equation}\label{mainsum1}
\left| S_1 \right|  \ll_m  \binom{n+c_2-1}{n} r^n \left(n\ln (n)+ \frac{2n}{c_1} \right) + \left(\frac{r}{n} \right)^{m+1} (c_2+m)_{m+1} (1-r)^{-c_2} (1+c_2 r).
\end{equation}
For $m=0$, the implicit constant in \eqref{mainsum1} is (at most) 24.
\end{lem}
\begin{proof}
We split $S_1=S'+S''$ into two parts, according to $i=n$ and $i<n$: 
\begin{align}
S' & = \beta^{n} \left( \sum_{k=m+1}^{n} \frac{\binom{n}{k}\binom{k-c_1}{k}}{\binom{n+c_1-1}{k}} \right) \frac{a^{(n)}(0)}{n!},\\
S'' &= \sum_{i=m+1}^{n-1} \beta^{i} \left( \sum_{k=m+1}^{i} \frac{\binom{i}{k}\binom{k-c_1}{k}}{\binom{n+c_1-1}{k}} \right) \frac{a^{(i)}(0)}{i!}.
\end{align}
By Lemmas \ref{lem2} and \ref{lem3}, we obtain
\begin{equation}\label{s2_1}
\begin{split}
\left| S' \right| &\le \beta^{n} \left(n\ln (n)+ \frac{2n}{c_1} \right) \frac{\alpha^{-n} (c_2+n-1)_n}{n!} = \binom{n+c_2-1}{n} r^n \left(n\ln (n)+ \frac{2n}{c_1} \right).
\end{split}
\end{equation}
If $m=n-1$, we have $S''=0$. Otherwise, Lemmas \ref{lem2} and \ref{lem3} give
\begin{equation}\label{withdenom}
\begin{split}
\left| S'' \right| &\le \frac{1}{(n+c_1-1)_{m+1}} \sum_{i=m+1}^{n-1} \beta^{i} (i)_{m+1}(i-m) \alpha^{-i} \binom{c_2+i-1}{i}\\
& = \frac{1}{(n+c_1-1)_{m+1}} \sum_{i=m+1}^{n-1} r^{i} (i)_{m+1}(i-m)  \binom{c_2+i-1}{i}.
\end{split}
\end{equation}
We bound the sum in \eqref{withdenom} from above by the infinite series
\begin{equation}\label{seriescomp}
\sum_{i=m+1}^{\infty} r^{i} (i)_{m+1}(i-m) \binom{c_2+i-1}{i}.
\end{equation}
We can compute \eqref{seriescomp}. We differentiate the identity
\begin{equation}
\sum_{i \ge 0} (-x)^i \binom{c_2+i-1}{i}=(1+x)^{-c_2}
\end{equation}
$m+1$ times, multiply by $x^{m+1}$ and obtain
\begin{equation}\label{toapplyop}
\sum_{i \ge m+1} (-x)^i (i)_{m+1} \binom{c_2+i-1}{i} = (-c_2)_{m+1} x^{m+1}(1+x)^{-c_2-m-1}.
\end{equation}
We apply the linear operator $f \mapsto x\cdot f'-m\cdot f$ to both sides of \eqref{toapplyop}:
\begin{equation}\label{afteroper}
\sum_{i \ge m+1} (-x)^i  (i)_{m+1} (i-m) \binom{c_2+i-1}{i} = (-c_2)_{m+1} x^{m+1}(1+x)^{-c_2-m-1} (1-(c_2+m+1)\frac{x}{1+x}).
\end{equation}
We plug $x=-r$ in \eqref{afteroper} and find the following formula for \eqref{seriescomp}:
\begin{equation}\label{eqimsum}
(-c_2)_{m+1}(-r)^{m+1} (1-r)^{-c_2-m-1} \left(1+\frac{r}{1-r}(c_2+m+1)\right).
\end{equation}
Dividing \eqref{eqimsum} by the denominator appearing in \eqref{withdenom}, we obtain the following upper bound on $\left| S'' \right|$:
\begin{equation}\label{almosts2_2}
\left| S'' \right| \le \frac{(c_2+m)_{m+1}}{(n+c_1-1)_{m+1}} \left(\frac{r}{1-r} \right)^{m+1}  (1-r)^{-c_2} \left(1+\frac{r}{1-r}(c_2+m+1)\right).
\end{equation}
Using \eqref{properties} and the estimate $(n+c_1-1)_{m+1} = \Omega_m(n^{m+1})$, we may simplify \eqref{almosts2_2} as
\begin{equation}\label{s2_2}
\left| S'' \right| \ll_m \left(\frac{r}{n} \right)^{m+1} (c_2+m)_{m+1} (1-r)^{-c_2} (1+c_2 r).
\end{equation} 
By combining \eqref{s2_1} and \eqref{s2_2}, we establish \eqref{mainsum1}.

Now we assume that $m=0$. If $n=1$, we have $S''=0$. Otherwise, plugging $m=0$ in \eqref{almosts2_2}, we obtain
\begin{equation}\label{summ0case}
\left| S'' \right| \le \frac{c_2r}{n} (1-r)^{-c_2} \frac{n}{n+c_1-1} \frac{1+c_2 r }{(1-r)^2}.
\end{equation}
Since $\frac{n}{n+c_1-1} \le 2$ and \eqref{properties} implies that $\frac{1+c_2r}{(1-r)^2} \le \frac{1+c_2r}{(1-\frac{1}{\sqrt{2}})^2} \le 12 (1+c_2 r)$, equation \eqref{summ0case} implies that
 \begin{equation}\label{implicitsumm0}
\left| S'' \right| \le 24\frac{c_2r}{n} (1-r)^{-c_2} (1+c_2 r).
\end{equation}
By \eqref{s2_1} and \eqref{implicitsumm0}, we find that the implicit constant in \eqref{mainsum1} is at most 24, as needed. 
\end{proof}
\subsection{Integral bound}
\begin{lem}\label{intlem}
Let $a(x)=\exp\left(\sum_{n \ge 1} \frac{\widetilde{a_n} x^n}{n} \right)=\sum a_n x^n$ be a power series satisfying \eqref{properties3} for some $\alpha, c_2 > 0$. Let $\beta>0$ be a real number such that $r=\beta/\alpha$ satisfies \eqref{properties}. Let $0\le m<n$. Define
\begin{equation}\label{defsinlem}
S_2 = \sum_{k=0}^{m} \frac{\binom{k-c_1}{k}}{\binom{n+c_1-1}{k}} \frac{\beta^{k}}{k!} \int_{0}^{\beta} \frac{(\beta-x)^{n-k}}{(n-k)!} a^{(n+1)}(x) dx.
\end{equation}
Then
\begin{equation}\label{firstpartint}
\left| S_2 \right| \ll_m \binom{n+c_2-1}{n} r^{n} (1+c_2r)(1-r)^{-c_2}.
\end{equation}
For $m=0$, the implicit constant in \eqref{firstpartint} is (at most) 12.
\end{lem}
\begin{proof}
In \eqref{defsinlem}, we replace $\binom{k-c_1}{k}$ by the upper bound $1$ and $\binom{n+c_1-1}{k}$ by the lower bound $\binom{n-1}{k}$:
\begin{align}\label{eq:integrandref}
\left| S_2 \right| & \le \sum_{k=0}^{m} \frac{1}{\binom{n-1}{k}} \frac{\beta^{n}}{k!} \int_{0}^{\beta} \frac{(1-\frac{x}{\beta} )^{n-k}}{(n-k)!}
\left|  a^{(n+1)}(x)\right|  dx .
\end{align}
We use Lemma \ref{lem3} to replace the number $|a^{(n+1)}(x)|$ appearing in the integrand of \eqref{eq:integrandref} with $(c_2+n)_{n+1} (1-\frac{x}{\alpha})^{-c_2-n-1} \alpha^{-n-1}$:
\begin{equation}\label{eq:integ1}
\begin{split}
\left| S_2 \right|& \le \sum_{k=0}^{m} \frac{1}{\binom{n-1}{k}} \frac{\beta^n}{k!} \int_{0}^{\beta} \frac{(1-\frac{x}{\beta})^{n-k}}{(n-k)!} (c_2+n)_{n+1}\alpha^{-n-1} (1-\frac{x}{\alpha})^{-c_2-n-1} dx \\
& = \sum_{k=0}^{m} \frac{(c_2+n)_{n+1}}{\binom{n-1}{k}k!(n-k)!} \beta^n \alpha^{-n-1} \int_{0}^{\beta} (1-\frac{x}{\beta})^{n-k}(1-\frac{x}{\alpha})^{-c_2-n-1} dx.
\end{split}
\end{equation}
We simplify the expression outside the integral:
\begin{equation}\label{eq:simploutint}
\begin{split}
\frac{(c_2+n)_{n+1}}{\binom{n-1}{k}k!(n-k)!} \beta^n \alpha^{-n-1} &= \binom{c_2+n}{n+1}(n+1)n \cdot r^{n+1} \cdot \frac{\beta^{-1}}{ n-k} \\
& = \binom{n+c_2-1}{n} n\left((n+c_2)r\right) \cdot r^n  \cdot  \frac{\beta^{-1}}{ n-k}\\
& \le \binom{n+c_2-1}{n} (1+c_2r)n^2 \cdot r^n  \cdot  \frac{\beta^{-1}}{ n-k}.
\end{split}
\end{equation}
Plugging \eqref{eq:simploutint} back in \eqref{eq:integ1}, we see
\begin{equation}\label{integbound}
\begin{split}
\left| S_2 \right| & \le \binom{n+c_2-1}{n}(1+c_2r)n^2 \cdot r^{n} \sum_{k=0}^{m} \int_{0}^{\beta} \frac{1}{n-k} (1-\frac{x}{\beta})^{n-k}(1-\frac{x}{\alpha})^{-c_2-n-1} \frac{dx}{\beta}.
\end{split}
\end{equation}
We perform the change of variables $s:=x/\beta$ in the right-hand side of \eqref{integbound}, and obtain
\begin{align}\label{eq:monoint}
\left|S_2 \right| & \le \binom{n+c_2-1}{n}(1+c_2r)n^2 \cdot r^{n} \sum_{k=0}^{m} \int_{0}^{1} \frac{(1-s)^{n-k}}{n-k}(1-rs)^{-c_2-n-1} ds.
\end{align}
We rewrite the integrand $\frac{(1-s)^{n-k}}{n-k}(1-rs)^{-c_2-n-1}$ as 
\begin{equation}\label{integs3rer}
\frac{1}{n-k} \left(\frac{1-s}{1-rs}\right)^{n-k} (1-rs)^{-c_2-k-1}.
\end{equation}
We apply two basic inequalities to \eqref{integs3rer}, which hold for $s\in [0,1]$: $0 \le \frac{1-s}{1-rs} \le 1+s(r-1)$ and $(1-rs)^{-1} \le (1-r)^{-1}$. This allows us to bound the right-hand side of \eqref{eq:monoint} from above by
\begin{equation}\label{beforeevalint}
\binom{n+c_2-1}{n} (1+c_2 r)n^2 \cdot r^{n}\sum_{k=0}^{m} \frac{1}{n-k} (1-r)^{-c_2-k-1} \int_{0}^{1} (1+s(r-1))^{n-k} ds.
\end{equation}
The integral in \eqref{beforeevalint} can be evaluated precisely as $\frac{1-r^{n-k+1}}{1-r}\frac{1}{n-k+1}$, which we bound from above by $\frac{1}{(1-r)(n-k+1)}$. Thus, we obtain
\begin{align}\label{eq:befmkrep}
\left| S_2 \right| & \le \binom{n+c_2-1}{n}(1+c_2r)n^2 \cdot r^{n} \sum_{k=0}^{m} \frac{1}{(n-k)(n-k+1)} (1-r)^{-c_2-k-2}.
\end{align}
Since $k \le m$, we may replace $(1-r)^{-c_2-k-2}$ in \eqref{eq:befmkrep} by $(1-r)^{-c_2-m-2}$:
\begin{align}\label{eq:totel}
\left| S_2 \right| & \le \binom{n+c_2-1}{n} (1+c_2r) n^2 \cdot r^{n} (1-r)^{-c_2-m-2} \sum_{k=0}^{m} \frac{1}{(n-k)(n-k+1)}.
\end{align}
The sum in \eqref{eq:totel} telescopes as follows:
\begin{equation}
\begin{split}
\sum_{k=0}^{m} \frac{1}{(n-k)(n-k+1)} & = \sum_{k=0}^{m} \left(\frac{1}{n-k} - \frac{1}{n-k+1}\right) \\
&= \frac{1}{n-m}-\frac{1}{n+1} = \frac{1+m}{(n-m)(n+1)}.
\end{split}
\end{equation}
Thus, \eqref{eq:totel} becomes
\begin{equation}\label{genm}
\left| S_2 \right| \le 	\binom{n+c_2-1}{n}(1+c_2r) \frac{n(1+m)}{n-m} r^{n} (1-r)^{-c_2-m-2}.
\end{equation}
Since $\frac{n(1+m)}{n-m} \le (1+m)^2 \ll_m 1$ and $r \le 1/\sqrt{2}$ by \eqref{properties}, the bound \eqref{genm} may be simplified as
\begin{equation}\label{intbeforesimpm}
\left| S_2 \right| \ll_m \binom{n+c_2-1}{n} r^{n} (1+c_2r)(1-r)^{-c_2},
\end{equation}
where the implied constant is $(1+m)^2(1-1/\sqrt{2})^{-m-2}$. This establishes \eqref{firstpartint}. For $m=0$, the implied constant is less than $12$.
\end{proof}
\subsection{Proof of Theorem \ref{thm1}}
\begin{proof}
Recall $b_n = [x^n] (1-x/\beta)^{-c_1} = \left(-1/\beta \right)^n \binom{-c_1}{n} = \beta^{-n} \binom{n+c_1-1}{n}$. We have
\begin{equation}\label{writedownbn}
\begin{split}
f_n &= \sum_{i=0}^{n} a_i b_{n-i} =  \sum_{i=0}^{n} \frac{a^{(i)}(0)}{i!} \cdot \left(-\frac{1}{\beta}\right)^{n-i} \binom{-c_1}{n-i} \\
&= \left(-\frac{1}{\beta}\right)^n\binom{-c_1}{n} \sum_{i=0}^{n} \left(-\frac{1}{\beta}\right)^{-i} \frac{\binom{-c_1}{n-i}}{\binom{-c_1}{n}} \frac{a^{(i)}(0)}{i!} \\
&= b_n \cdot \left( \sum_{i=0}^{n} \beta^{i} (-1)^i\frac{\binom{-c_1}{n-i}}{\binom{-c_1}{n}} \frac{a^{(i)}(0)}{i!} \right).
\end{split}
\end{equation}
By Lemma \ref{lem1}, we may replace $(-1)^i\frac{\binom{-c_1}{n-i}}{\binom{-c_1}{n}}$ in \eqref{writedownbn} with $\sum_{k=0}^{i} \frac{\binom{i}{k}\binom{k-c_1}{k}}{\binom{n+c_1-1}{k}}$ and split the sum according to small and large $k$:
\begin{equation}\label{struct1}
\begin{split}
f_n &= b_n \cdot \sum_{i=0}^{n} \beta^{i} \left( \sum_{k=0}^{i} \frac{\binom{i}{k}\binom{k-c_1}{k}}{\binom{n+c_1-1}{k}} \right)\frac{a^{(i)}(0)}{ i!} \\
&=    b_n \cdot \sum_{i=0}^{n} \beta^{i} \left( \sum_{k=0}^{m} \frac{\binom{i}{k}\binom{k-c_1}{k}}{\binom{n+c_1-1}{k}} \right) \frac{a^{(i)}(0)}{i!} + b_n \cdot \sum_{i=m+1}^{n} \beta^{i} \left( \sum_{k=m+1}^{i} \frac{\binom{i}{k}\binom{k-c_1}{k}}{\binom{n+c_1-1}{k}} \right) \frac{a^{(i)}(0)}{ i!}. \end{split}
\end{equation}
We denote the first sum by $S_0$ and the second sum by $S_1$:
\begin{equation}
\begin{split}
S_0 &=\sum_{i=0}^{n} \beta^{i} \left( \sum_{k=0}^{m} \frac{\binom{i}{k}\binom{k-c_1}{k}}{\binom{n+c_1-1}{k}} \right) \frac{a^{(i)}(0)}{i!},\\
S_1 &=  \sum_{i=m+1}^{n} \beta^{i} \left( \sum_{k=m+1}^{i} \frac{\binom{i}{k}\binom{k-c_1}{k}}{\binom{n+c_1-1}{k}} \right) \frac{a^{(i)}(0)}{ i!}.
\end{split}
\end{equation} 
We rearrange the terms of $S_0$ as follows:
\begin{equation}\label{s1rear}
\begin{split}
S_0&=  \sum_{k=0}^{m} \frac{\binom{k-c_1}{k}}{\binom{n+c_1-1}{k}} \sum_{i=k}^{n} \binom{i}{k} \frac{a^{(i)}(0)}{i!}\beta^{i} \\
&= \sum_{k=0}^{m} \frac{\binom{k-c_1}{k}}{\binom{n+c_1-1}{k}} \frac{\beta^{k}}{k!} \sum_{j=0}^{n-k}\frac{a^{(j+k)}(0)}{j!}\beta^{j}.
\end{split}
\end{equation}
Let
\begin{equation}
\begin{split}
M &= \sum_{k=0}^{m} \frac{\binom{k-c_1}{k}}{\binom{n+c_1-1}{k}} \frac{\beta^{k}}{k!} a^{(k)}(\beta), \\
S_2 &= \sum_{k=0}^{m} \frac{\binom{k-c_1}{k}}{\binom{n+c_1-1}{k}} \frac{\beta^{k}}{k!} \int_{0}^{\beta} \frac{(\beta-x)^{n-k}}{(n-k)!} a^{(n+1)}(x) dx.
\end{split}
\end{equation}
By the integral formula for the remainder in a Taylor series, we may replace $\sum_{j=0}^{n-k}\frac{a^{(j+k)}(0)}{j!}\beta^{j}$ in \eqref{s1rear} with
\begin{equation*}
a^{(k)}(\beta) -  \int_{0}^{\beta} \frac{(\beta-x)^{n-k}}{(n-k)!} a^{(n+1)}(x) dx,
\end{equation*}
and obtain
\begin{equation}\label{eq:struct2}
\begin{split}
S_0 &= M-S_2.
\end{split}
\end{equation}
Summarizing \eqref{struct1} and \eqref{eq:struct2}, we see
\begin{align}\label{eq:tocontcor}
f_n &= b_n \cdot (M + S_1-S_2).
\end{align}
Note that $b_n \cdot M$ is the main term of \eqref{edef}. The expression $E$, as defined in \eqref{edef}, equals $S_1-S_2$. The sums $S_1$, $S_2$ are bounded in Lemmas \ref{sumlem} and \ref{intlem}, respectively, and these bounds give
\begin{equation}\label{eq:o_dfin1}
\begin{split}
\left| E \right| &\ll_m  \left(\frac{r}{n} \right)^{m+1} (1-r)^{-c_2} (1+c_2r)(c_2+m)_{m+1}   \\
& \quad+ \binom{n+c_2-1}{n}r^n \left(n\ln (n)+ \frac{2n}{c_1} + (1-r)^{-c_2} \left(1+c_2 r\right) \right).
\end{split}
\end{equation}
We simplify \eqref{eq:o_dfin1} using the following two inequalities:
\begin{equation}\label{eqsimpcomb1}
1+c_2 r \le \exp(c_2 r),
\end{equation}
and, since $-\ln(1-r) \le 2r$ for $0\le r \le 1/\sqrt{2}$ by simple calculus, we have
\begin{equation}\label{eqsimpcomb2}
(1-r)^{-c_2} = \exp( -c_2\ln(1-r)) \le \exp (2c_2 r).
\end{equation}
Plugging \eqref{eqsimpcomb1} and \eqref{eqsimpcomb2} into \eqref{eq:o_dfin1}, we obtain 
\begin{equation}\label{finalthm1}
\begin{split}
\left| E \right| & \ll_m  \left(\frac{r}{n} \right)^{m+1} \exp(3 c_2 r) (c_2+m)_{m+1} \\
& \qquad + \binom{n+c_2-1}{n}r^n \left(n\ln (n)+ \frac{2n}{c_1} + \exp(3c_2 r) \right).
\end{split}
\end{equation}
We note that
\begin{equation}\label{lastsimpli}
n\ln (n)+ \frac{2n}{c_1} + \exp(3c_2 r)  \le \exp(3c_2r) \frac{4n^2}{c_1}.
\end{equation}
Plugging \eqref{lastsimpli} in\eqref{finalthm1}, we obtain
\begin{equation}\label{finalfinalthm1}
\begin{split}
\left| E \right| & \ll_m  \exp(3 c_2 r) \left( \left(\frac{r}{n} \right)^{m+1}  (c_2+m)_{m+1} + \binom{n+c_2-1}{n} \frac{4n^2}{c_1} r^n  \right),
\end{split}
\end{equation}
as needed. For $m=0$, the implied coefficient is the maximum of the two implied constants of Lemmas \ref{sumlem} and \ref{intlem}, i.e. 24.
\end{proof}
\subsection{Proof of Corollary \ref{nicer}}
We apply Theorem \ref{thm1} with $m=0$, and obtain
\begin{equation}\label{twoerrorcor}
\begin{split}
\left| E \right| & \ll  \exp(3 c_2 r) \left( \frac{c_2r}{n}+ \binom{n+c_2-1}{n} \frac{4n^2}{c_1} r^n  \right),
\end{split}
\end{equation}
where the implied constant is 24. We find a range of $n$ where the first term of the right-hand side of \eqref{twoerrorcor} dominates the error, i.e. we want to solve
\begin{equation}\label{dominanteineq}
\frac{c_2r}{n} \ge  \binom{n+c_2-1}{n} \frac{4n^2}{c_1} r^n .
\end{equation}
For $n$ for which \eqref{dominanteineq} holds, \eqref{expcor} holds with an  implied constant (at most) $24+24=48$. We find a range of $n$ for which \eqref{dominanteineq} holds by making several simplifications. Note that
\begin{equation}
\begin{split}
\binom{n+c_2-1}{n}n^2 &= \binom{n+c_2-1}{n-1}c_2 n  = \prod_{i=1}^{n-1} \left(1 + \frac{c_2}{i} \right) c_2 n\\
& \le  \exp \left( c_2 \sum_{i=1}^{n-1} \frac{1}{i} \right) c_2 n  \le \exp \left( c_2 \left( \ln(n) + 1 \right) \right) c_2n \\
& = (ne)^{c_2+1} \frac{c_2}{e},
\end{split}
\end{equation}
so we may replace \eqref{dominanteineq} by the following stricter inequality:
\begin{equation}
\frac{r}{n} \ge (ne)^{c_2+1} \frac{1}{c_1} \frac{4}{e}r^n, 
\end{equation}
or, equivalently, 
\begin{equation}\label{toreplacebytwo}
\left(\frac{1}{r} \right)^{n-1} \ge \left( ne \right)^{c_2+2} \frac{1}{c_1} \frac{4}{e^2}.
\end{equation}
Inequality \eqref{toreplacebytwo} holds when the following two inequalities hold simultaneously: 
\begin{equation}\label{replacetwo}
\left(\frac{1}{r} \right)^{(n-1)/2 } \ge \left( ne\right)^{c_2+2}, \quad \left(\frac{1}{r} \right)^{(n-1)/2 } \ge \frac{1}{c_1}.
\end{equation}
The first inequality of \eqref{replacetwo} holds whenever $\frac{n-1}{\ln(n)+1} \ge \frac{2c_2+4}{\ln(\frac{1}{r})}$, which, by Lemma \ref{lemmaforcor}, is satisfied when
\begin{equation}
n \ge 1+5\left(\frac{2c_2+4}{\ln(\frac{1}{r})} + 1\right) \ln \left(\frac{2c_2+4}{\ln(\frac{1}{r})} + 1\right).
\end{equation}
The second inequality of \eqref{replacetwo} holds whenever $n \ge 2\frac{\ln( c_1^{-1})}{\ln(\frac{1}{r})} + 1$. 
This establishes Corollary \ref{nicer}.

\section*{Acknowledgments}
I thank my advisor, Lior Bary-Soroker, for introducing me to Landau's theorem, for helpful discussions and for teaching me, patiently, how to write. Also, I thank Amotz Oppenheim and Ze\'ev Rudnick for their useful comments on this paper, Mikhail Sodin for a helpful discussion related to the analytic part of the paper, and the anonymous referee for her/his valuable remarks. This research was partially supported by the Israel Science Foundation grants no. 952/14 and no. 382/15.
\bibliographystyle{alpha}
\bibliography{references}

Raymond and Beverly Sackler School of Mathematical Sciences, Tel Aviv University,
P. O. Box 39040, Tel Aviv 6997801, Israel.
E-mail address: ofir.goro@gmail.com

\end{document}